\documentclass{amsart}
\usepackage[utf8]{inputenc}
\usepackage{amsmath}
\usepackage{amsfonts}
\usepackage{amssymb}
\usepackage{amsthm}
\usepackage{verbatim}
\usepackage[]{graphicx}
\usepackage[ruled,vlined]{algorithm2e}
\SetKw{Break}{break}
\usepackage{xcolor}
\usepackage{todonotes}
\usepackage{hyperref}
\usepackage[style=alphabetic,giveninits,citestyle=alphabetic,doi=false,isbn=false,url=false,eprint=false,maxnames=50,sorting=nyt]{biblatex}
\hypersetup{
    citecolor=green,
    colorlinks=true,
    linkcolor=blue,
    filecolor=magenta,      
    urlcolor=red,
}

\newtheorem{theorem}{Theorem}[section]

\newtheorem{lemma}[theorem]{Lemma}
\newtheorem{proposition}[theorem]{Proposition}

\newtheorem{remark}[theorem]{Remark}

\newcommand{\dee}{{\rm d}}

\newcommand{\R}{\mathbb{R}}

\newcommand{\boundaryFormOfDerivative}{\tilde \xi_f}
\newcommand{\DiscreteboundaryFormOfDerivative}{\tilde \xi_{f,h}}
\newcommand{\SecondDiscreteBoundaryFormOfDerivative}{\bar \xi_{f,h}}

\DeclareMathOperator{\Div}{div}

\DeclareMathOperator*{\argmin}{arg\,min}

\title[$W^{1,\infty}$ shape optimisation with Lipschitz domains]{A novel $W^{1,\infty}$ approach to shape optimisation with Lipschitz domains}
\author[K. Deckelnick]{Klaus Deckelnick}
\address{Institut f\"ur Analysis und Numerik, Otto-von-Guericke-Universit\"at Magdeburg, 39106 Magdeburg, Germany}
\email{Klaus.Deckelnick@ovgu.de}

\author[P. J. Herbert]{Philip J. Herbert}
\address{Mathematisches Institut,
Campus Koblenz,
Universit\"at Koblenz-Landau,
Universit\"atsstr. 1,
D-56070 Koblenz, Germany}
\email{pherbert@uni-koblenz.de}
\author[M. Hinze]{Michael Hinze}
\address{Mathematisches Institut,
Campus Koblenz,
Universit\"at Koblenz-Landau,
Universit\"atsstr. 1,
D-56070 Koblenz, Germany}
\email{hinze@uni-koblenz.de}
\date{\today}

\begin{document}

\maketitle
\begin{abstract}
    This article introduces a novel method for the implementation of shape optimisation with Lipschitz domains. We propose to use the shape derivative to determine deformation fields which represent steepest descent directions of the shape functional in the $W^{1,\infty}-$ topology. The idea of our approach is demonstrated for shape optimisation of $n$-dimensional star-shaped domains, which we represent as functions defined on the unit $(n-1)$-sphere. In this setting we provide the specific form of the shape derivative and prove the existence of solutions to the underlying shape optimisation problem. Moreover, we show the existence of a direction of steepest descent in the $W^{1,\infty}-$ topology. We also note that shape optimisation in this context is closely related to the $\infty-$Laplacian, and to optimal transport, where we highlight the latter in the numerics section. We present several numerical experiments in two dimensions illustrating that our approach seems to be superior over a widely used Hilbert space method in the considered examples, in particular in developing optimised shapes with corners.
\end{abstract}
\section{Introduction}
In the present work we are interested in the numerical solution of a certain class of shape optimisation problems
\begin{displaymath}
\min \mathcal J(\Omega), \; \Omega \in \mathcal S,
\end{displaymath}
where $\mathcal S$ denotes the set of admissible shapes to be specified in the respective application. A common approach in order to calculate at least local minima of
$\mathcal J$ consists in applying the steepest descent method by using the shape derivative of $\mathcal J$. More precisely, given a shape $\Omega \in \mathcal S$, one determines a
descent vector field $V^*\colon \mathbb R^n \rightarrow \mathbb R^n$ which is Lipschitz and satisfies $\mathcal J'(\Omega)(V^*)<0$ \cite[Section 5.2]{HenPie18}, then one sets $\Omega_{\mbox{new}}:=(\mbox{id} + \alpha V^*)(\Omega)$ for
a suitable step size $\alpha>0$.
A common approach in order to determine a descent direction $V^*$ employs a Hilbert space setting. Let $H$ be a Hilbert space
with scalar product $a(\cdot,\cdot)$, then $V$ is determined by minimising
\begin{displaymath}
V \mapsto a(V,V) + \mathcal J'(\Omega)(V), V \in H.
\end{displaymath}
A nice discussion of the pros and cons of this approach can be found in Section 5.2 of \cite{ADJ21}.
Typical choices of $H$ are the Sobolev spaces $H^m(\mathbb R^n;\mathbb R^n)$, where
one however needs to choose $m$ sufficiently large in order to obtain a Lipschitz transformation. Ideally, one would like to
determine a Lipschitz vector field $V^*$ such that $\| V^* \|_{W^{1,\infty}(\Omega;\R^n)}\leq 1$ and
\begin{displaymath}
\mathcal J'(\Omega)(V^*) = \min_{\| V \|_{W^{1,\infty}(\Omega;\R^n)} \leq 1} \mathcal J'(\Omega)(V).
\end{displaymath}
However, compared to a Hilbert space, the functional analytic properties of the space of Lipschitz functions are less favourable making
this approach difficult both from an analytical and a computational point of view. In this paper we aim to address this task 
in the special case  that the admissible domains
$\Omega \subset \mathbb R^n$ are star-shaped with respect to the origin so that shapes and their perturbations can be described in terms of scalar functions
$f\colon \mathbb S^{n-1}:= \lbrace x \in \mathbb R^n \, | \, |x|=1 \rbrace \rightarrow \mathbb R$.
The restriction to star-shaped domains allows for a deeper analysis at the expense of generality and is regularly considered in shape optimisation, e.g. \cite{EppHarRei07}, 
\cite{BouChSa20}. 

In the described setting we consider the following model problem 
\begin{equation} \label{model}
\inf_{\Omega \in \mathcal S} \mathcal J(\Omega):= \frac{1}{2} \int_\Omega | u_\Omega - z |^2 \dee x,
\end{equation}
where $u_\Omega \in H^1_0(\Omega)$ is the unique weak solution of 
\begin{equation} \label{state}
\int_\Omega \nabla u_\Omega \cdot \nabla \eta \, \dee x = \int_\Omega F \,  \eta \, \dee x \qquad \forall \eta \in H^1_0(\Omega)
\end{equation}
and $z \in H^1(D), \, F \in L^2(D)$ are given functions on some hold--all domain $D\subset \R^n$.
We note that this energy and PDE problem are probably the most simple example of PDE constrained shape optimisation.
Some simple extensions might be to consider higher powers in the integrand of the energy, or parts of the domain which obey a Neumann boundary condition.
It is expected that the strategy we consider may be applied to more general second order PDE constrained shape optimisation problems, in particular problems from elasticity.
In Section 2 we reformulate  \eqref{model} as a minimisation problem
on a suitable subset of $W^{1,\infty}(\mathbb S^{n-1})$  and calculate the shape derivative in terms of the solutions of the state and adjoint equations.
Furthermore, we prove the existence of an optimal Lipschitz--continuous descent direction, for which we
derive an explicit formula in the case $n=2$. Using a discrete version of this formula together with finite element discretisations of the state and adjoint equation 
we obtain an approximation of the optimal descent direction which is used in the steepest descent method. The numerical experiments, which are all two-dimensional, shown in Section 4 demonstrate that
this novel approach performs better than a very typical method which relies on  $H^1$--regularisation. Let us also mention that our approach may be related to optimal transport,
see \cite{San15}. \\\\
There exists a vast amount of literature related to shape optimisation problems. We first mention the seminal works of Delfour and Zol\'esio \cite{DelZol11},  of Sokolowski and Zol\'esio \cite{SZ92}, and the recent overview article \cite{ADJ21} by Allaire, Dapogny, and Jouve, where also a comprehensive  bibliography on the topic can be found. The mathematical and numerical analysis of shape optimisation problems has a long history, see e.g. \cite{BFLS97,GM94,MS76,S80}. With increasing computing power, shape optimisation has experienced a renaissance in recent years \cite{SSW15,SSW16,SW17}, especially in fluid mechanical applications \cite{BLUU09,FLUU17,GHHK15,GHKL18,HRHD18,HUU20,HSU21,KMHR19,SISG13}. A steepest descent method for the numerical solution utilising a Hilbert-space framework for PDE constrained shape optimisation is investigated in \cite{HP15}. A comparison of numerical approximations of Hilbertian shape gradients in boundary and volume form is presented in \cite{HPS15}.
A particularly interesting Hilbertian method is considered in \cite{IglSturWec18}, based on Cauchy-Riemann equations, where the authors have an example which is able to form corners.
A downside they mention is that the method is quite specific to two-dimensional shapes.
A specific choice of Hilbert space would be reproducing kernel Hilbert spaces which have been considered in \cite{EigStu18}, where an explicit form of the gradient is shown for certain kernels.
Finally we recall that an extensive summary of the state of the art in numerical approaches to shape and topology optimisation is given in \cite[Chapter 6-9]{ADJ21}. 

\section{Analysis of a model problem}
Let us begin by introducing some notation.
The space of Lipschitz functions on $\mathbb{S}^{n-1}$ is given as
\[
    C^{0,1}(\mathbb{S}^{n-1}) := \left\{ u \colon \mathbb{S}^{n-1} \to \R \,|\, \sup_{x,y \in \mathbb{S}^{n-1} ,\, x \neq y} \frac{|u(x)-u(y)|}{d(x,y)} < \infty \right\},
\]
where $d\colon \mathbb{S}^{n-1}\times \mathbb{S}^{n-1} \to \R$ is the intrinsic metric on $\mathbb{S}^{n-1}$.
One may equivalently define $C^{0,1}(\mathbb S^{n-1})$ using the standard Euclidean distance in the semi-norm.
We will be using Lebesgue and Sobolev spaces on $\mathbb{S}^{n-1}$, equipped with the $(n-1)$-dimensional Hausdorff measure on $\mathbb{S}^{n-1}$.
Since $C^{0,1}(\mathbb S^{n-1})\cong W^{1,\infty}(\mathbb S^{n-1})$ \cite{EvaGar15}, the tangential gradient $\nabla_T f$ is defined almost everywhere on $\mathbb S^{n-1}$.
We give the explicit definition of the tangential gradient by its definition on charts.
Let $\Theta \subset \R^{n-1}$ be open and bounded and $X \colon \Theta \to \mathbb{S}^{n-1}$ be a $C^2$-diffeomorphism onto its image, $U:= X(\Theta)$.
Then, for almost every $\omega \in U$,
\begin{displaymath}
\nabla_T f(\omega) := \left( \sum_{i,j=1}^{n-1} g^{ij} \frac{\partial (f\circ X)}{\partial \theta_j} \frac{\partial X}{\partial \theta_i} \right) \circ X^{-1}(\omega), 
\end{displaymath}
where $\{\theta_i\}_{i=1}^{n-1}$ are coordinates on $\Theta$ and $g^{ij}$ is the $ij$ element of the inverse matrix of $G$, which has elements $g_{ij} = \frac{\partial X}{\partial \theta_i}\cdot \frac{\partial X}{\partial \theta_j}$ for $i,j = 1,...,n-1$.
For more details on this parametric representation, see \cite{DecDziEll05}, in particular equation (2.14).
We note that this definition is independent of the paramaterisation $X$ as well as for $f \in W^{1,\infty}(\mathbb{S}^{n-1})$
\begin{equation} \label{nablat}
\nabla _T f \in L^\infty(\mathbb S^{n-1}), \quad \nabla_T f(\omega) \cdot \omega =0 \mbox{ a.e. on } \mathbb S^{n-1}.
\end{equation}
\subsection{Reformulation and existence of a minimum}
A bounded domain $\Omega \subset \mathbb R^n$ is star--shaped with respect to the origin if $[0,x] \subset \Omega$ for every $x \in \Omega$, where for $x,y \in \R^n$, $[x,y]:= \{ x + t(y-x) \in \R^n : t \in [0,1] \}$.
Furthermore, $\Omega$ is called star--shaped with respect to $B_\epsilon(0)$ if $[y,x] \subset \Omega$ for every $y \in B_\epsilon(0)$ and every $x \in \Omega$. 
For a bounded domain $\Omega$ that is star--shaped with respect to the
origin we denote by $f_\Omega\colon \mathbb S^{n-1} \rightarrow \mathbb R_{>0}$ its radial function given by
\begin{equation} \label{radial}
f_{\Omega}(\omega):= \sup \lbrace \lambda >0 \, | \, \lambda \omega \in \Omega \rbrace, \quad \omega \in \mathbb S^{n-1}.
\end{equation}
It is shown in \cite[Lemma 2, Section 3.2]{Bur98} that if $\Omega$ is star--shaped with respect to $0$, then it is star--shaped with respect to a ball $B_\epsilon(0)$  if and only if  $f_\Omega \in C^{0,1}(\mathbb{S}^{n-1})$.
We wish to show that a given positive function which has bounded tangential gradient generates a domain which is star--shaped with respect to a ball.
In order to formulate this, we assign to a positive function $f\colon\mathbb S^{n-1} \rightarrow \mathbb R$ the set 
\begin{equation} \label{omegaf}
\Omega_f:= \lbrace x \in \mathbb R^n \, | \, x=0 \mbox{ or } |x| < f(\omega_x), x \neq 0 \rbrace, \quad \mbox{ where } \omega_x=\frac{x}{|x|}.
\end{equation}
Clearly, $f$ is the radial function of $\Omega_f$ and, if $\Omega$ is star--shaped, then $\Omega_{f_\Omega} = \Omega$.

\begin{lemma} \label{starshaped}
Given $f \in W^{1,\infty}(\mathbb S^{n-1})$ with $\displaystyle f_0:= \min_{\omega \in \mathbb S^{n-1}} f(\omega)>0$ and $L := \Vert \nabla_{T} f \Vert_{L^{\infty}(\mathbb{S}^{n-1})}$. Then:
\\[1.2mm]
(i) The function $f$ satisfies 
\begin{equation}\label{eq:LipschitzBoundProof}
    |f(\omega_2) - f(\omega_1)| \leq L d(\omega_1,\omega_2) \leq L\frac{\pi}{2}|\omega_1-\omega_2|.
\end{equation}
(ii) $\Omega_f$ is star--shaped with respect to  $B_\epsilon(0)$, where $\epsilon = \frac{f_0^2}{{L\pi + f_0}}$. \\[1.2mm]
(iii) Let  $\Phi_f\colon \mathbb R^n \rightarrow \mathbb R^n$ be defined by
\begin{equation} \label{defPhi}
\Phi_f(x):= 
\left\{
\begin{array}{cl}
f (\omega_x) x, & x \neq 0, \\
0, & x=0.
\end{array}
\right.
\end{equation}
Then $\Phi_f$ is bi--Lipschitz from $B$ onto $\Omega_f$, where $B=\lbrace x \in \mathbb R^n \, | \, |x|<1 \rbrace$. In addition
\begin{equation} \label{phideriv}
D \Phi_f(x) = f (\omega_x) I + \omega_x  \otimes \nabla_{T} f (\omega_x) \; \mbox{ and } \;
\det D \Phi_f(x)=  f (\omega_x)^n  \; \mbox{ a.e. in } B,
\end{equation}
where $(a \otimes b)_{ij} := a_i b_j$ for vectors $a,b \in \R^n$.
\end{lemma}
\begin{proof} (i)
For $\delta \in (0,1)$, let us extend $f$ to the open set
$S_\delta := \{ x \in \R^n : |x| \in (1-\delta,1+\delta)\}$
via $\hat{f}(x):= f(\omega_x)$.
It is clear that $\hat{f} \in C^{0,1}(S_\delta)\cong W^{1,\infty}(S_\delta)$ with weak derivative $\nabla \hat{f}(x) = \frac{1}{|x|}\nabla_T f(\omega_x)$ for a.e. $x \in S_\delta$.
In particular
\[
    |\nabla \hat f(x) | = \frac{1}{|x|} | \nabla_T f(\omega_x)| \leq \frac{L}{1-\delta},\, x \in S_\delta.
\]
Denoting by $\hat f_\rho$ the standard mollification of $\hat{f}$ we have that for $\rho < \frac{\delta}{2}$, $\hat{f}_\rho \in C^\infty(S_{\frac{\delta}{2}})$ and $|\nabla \hat f_\rho| \leq \frac{L}{1-\delta}$ on $S_{\frac{\delta}{2}}$ as well as $\hat f_\rho \to \hat f = f$ uniformly on $\mathbb S^{n-1}$ as $\rho \to 0$.
Let $\omega_1, \omega_2 \in \mathbb S^{n-1}$ and $\eta \colon[0,1] \rightarrow \mathbb S^{n-1}$
be a curve with $\eta(0)=\omega_1, \eta(1)=\omega_2$ and $\int_0^1 | \eta'(t) | \, \dee t = d(\omega_1,\omega_2)$, where $d(\cdot,\cdot)$ is again the spherical metric on $\mathbb S^{n-1}$.
Then
\begin{displaymath}
\left| \hat f_\rho (\omega_2) - \hat f_\rho (\omega_1) \right| = \left| \int_0^1 \frac{d}{dt} (\hat f_\rho \circ \eta)(t) \dee t \right| = \left| \int_0^1 \nabla \hat f_\rho (\eta(t)) \cdot \eta'(t) \dee t \right|
\leq \frac{L d(\omega_1,\omega_2)}{1-\delta}.
\end{displaymath}
By first letting $\rho \to 0$ and afterwards $\delta \to 0$ we deduce \eqref{eq:LipschitzBoundProof}, observing that $d(\omega_1,\omega_2) \leq \frac{\pi}{2}|\omega_1-\omega_2|$.\\[1.2mm]
(ii)
Let $x \in \Omega_f \setminus \{0\}$, $y \in B_\epsilon(0)$ and $z = t x + (1-t)y$ for some $t \in (0,1)$.
We abbreviate $\tilde x = tx$ and have
\[
    |\omega_{\tilde {x} } - \omega_z|
    =
    \left| \frac{ \tilde{x}}{|\tilde{x}|} - \frac{z}{|z|}\right|
    \leq
    \frac{1}{|\tilde x|} |\tilde x - z| + |z| \left| \frac{1}{|z|} - \frac{1}{|\tilde x|}\right|
    \leq \frac{2}{|\tilde x| }|\tilde x - z|,
\]
which combined with \eqref{eq:LipschitzBoundProof} yields
\begin{displaymath}
    f(\omega_{\tilde {x}})
    \leq
    |f(\omega_{\tilde {x}}) - f(\omega_z)| + f(\omega_{z})
    \leq
    L \frac{\pi}{2}|\omega_{\tilde{x}} - \omega_z| + f(\omega_z)
    \leq
    \frac{L\pi}{|\tilde{x}|}|\tilde{x} - z| + f(\omega_z).
\end{displaymath}
Since $\omega_{\tilde{x}} = \omega_x$ and $\frac{|\tilde{x}|}{f(\omega_{\tilde{x}})} = \frac{t|x|}{f(\omega_x)} < t$ we obtain
\begin{align*}
    |z|
    \leq&
    |z-\tilde{x}| + |\tilde{x}|
    =
    |z-\tilde{x}| + \frac{|\tilde{x}|}{f(\omega_{\tilde{x}} )} f(\omega_{\tilde{x}} )
    <
    |z-\tilde{x}| + \frac{L \pi}{f(\omega_{\tilde{x}})} |\tilde x - z| + t f(\omega_z)
    \\
    \leq&
    (1-t)|y| \left( 1 + \frac{L \pi}{f_0} \right) + t f(\omega_z)
    \leq
    (1-t) \epsilon \frac{f_0 + L\pi}{f_0} + t f(\omega_z)
    \\
    \leq&
    (1-t)f_0 + t f(\omega_z)
    \leq f(\omega_z),
\end{align*}
therefore $z \in \Omega_f$ by the choice of $\epsilon$. \\[1.2mm]
(iii) Since $f(\omega) \geq f_0$ for all $\omega \in \mathbb S^{n-1}$ it is straightforward to verify that $\Phi_f$ is bi--Lipschitz from $B$ to $\Omega_f$. 
Furthermore, for $x \neq 0$
\begin{displaymath}
D \Phi_f(x) = f (\omega_x) I + \omega_x  \otimes P(x) \nabla_{T} f (\omega_x),
\end{displaymath}
where $P(x):=I- \omega_x \otimes \omega_x$ is the projection onto the tangent space $T_{\omega_x}\mathbb{S}^{n-1}$.
Observing that $\nabla_{T} f (\omega_x) \cdot x=0$ by \eqref{nablat} gives that $P(x) \nabla_T f(\omega_x) = \nabla_T f(\omega_x)$ to conclude the form of $D\Phi_f$.
Using that $\omega_x \otimes \nabla_T f(\omega_x)$ is a rank 1 term with vanishing trace, we deduce \eqref{phideriv}.
\end{proof}
\begin{remark}
Lemma 2 of \cite[Chapter 3]{Bur98} shows that one may take $\epsilon = \frac{2}{\pi}\frac{f_0^2}{\sqrt{L^2 + f_0^2}}$.
This value is different to that which we have considered in (ii).
\end{remark}

\noindent
Using \eqref{phideriv} together with a change of variables we infer that
\begin{equation} \label{volomega}
| \Omega_f | = \int_B | \mbox{det} D \Phi_f(x) | \, \dee x = \int_0^1 \int_{\mathbb S^{n-1}} f(\omega)^n \dee o_\omega r^{n-1} \dee r = \frac{1}{n} \int_{\mathbb S^{n-1}} f(\omega)^n \dee o_\omega,
\end{equation}
where we write $\dee o_\omega$ to be the surface element on $\mathbb{S}^{n-1}$.
Let us fix $\rho>0, L>0$ and $\gamma>0$ with $\gamma >  \rho^n | \mathbb S^{n-1} |$.  We define
\begin{multline} \label{defF}
\mathcal F:= \lbrace f \in W^{1,\infty}(\mathbb S^{n-1})  \, | \, f \geq \rho \mbox{ in } \mathbb S^{n-1}, \\ \Vert \nabla_T f \Vert_{L^\infty(\mathbb{S}^{n-1})} \leq L, \int_{\mathbb S^{n-1}} f(\omega)^n
\dee o_\omega = \gamma \rbrace.
\end{multline}
Note that if $f \in \mathcal F$, then there exists $\bar \omega \in \mathbb S^{n-1}$ such
that $f(\bar \omega)^n | \mathbb S^{n-1} | = \gamma$. By using the Lipschitz bound \eqref{eq:LipschitzBoundProof} we obtain for every $\omega \in \mathbb S^{n-1}$ that
\begin{equation}\label{eq:DefnOfR}
f(\omega)  \leq  f(\bar \omega) + L d(\omega, \bar \omega) \leq \bigl( | \mathbb S^{n-1} |^{-1} \gamma \bigr)^{\frac{1}{n}} + \pi L=:R,
\end{equation}
so that all sets $\Omega_f$ given by \eqref{omegaf} are contained in the hold--all domain $D:=B_R(0)$. Given $F,z \in L^2(D)$, we now define
\begin{displaymath}
J\colon \mathcal F \rightarrow \mathbb R, \,  J(f):= \mathcal J(\Omega_f)= \frac{1}{2} \int_{\Omega_f} |u - z|^2 \dee x,
\end{displaymath}
where $u \in H^1_0(\Omega_f)$ solves
\begin{equation} \label{statef}
\int_{\Omega_f} \nabla u \cdot \nabla \eta \, \dee x = \int_{\Omega_f} F \,  \eta \, \dee x \qquad \forall \eta \in H^1_0(\Omega_f).
\end{equation}
Hence we consider the optimisation problem \eqref{model}, \eqref{state} in the class $\mathcal S = \lbrace \Omega_f \, | \, f \in \mathcal F \rbrace$. In view of 
Lemma \ref{starshaped} and \eqref{volomega} the class of admissible domains consists of bounded domains of fixed volume, which contain $B_\rho(0)$ and which are star--shaped with respect to 
$B_\epsilon(0)$, where $\epsilon=\frac{\rho^2}{{L\pi + \rho}}$. 
Let us next establish the existence of a solution of the resulting optimisation problem. 

\begin{theorem}\label{thm:ExistsConditionalMinimiser}
There exists $f_* \in \mathcal F$ such that $J(f_*)=\min_{f \in \mathcal F} J(f)$.
\end{theorem}
\begin{proof}
Since $\gamma>\rho^n | \mathbb S^{n-1} |$, the function $f := \bigl( | \mathbb S^{n-1} |^{-1} \gamma \bigr)^{\frac{1}{n}}$ belongs to $\mathcal F$, so that  $\mathcal F$ is non--empty. 
Let $(f_k)_{k \in \mathbb N} \subset \mathcal F$ be a sequence such that $J(f_k) \searrow \inf_{f \in \mathcal F} J(f)$.
By the theorem of Arzel\`a-Ascoli and the fact that bounded sequences in $L^\infty$ contain weak-$*$ convergent subsequences, one has that there exists a subsequence, again denoted by $(f_k)_{k \in \mathbb N}$ and $f_* \in W^{1,\infty}(\mathbb S^{n-1})$ such that
\begin{displaymath}
f_k \rightarrow f_* \mbox{ in } C(\mathbb S^{n-1}) \mbox{ and } \nabla_{T} f_k \overset{*}{\rightharpoonup} \nabla_{T} f_* \mbox{ in } L^{\infty}(\mathbb S^{n-1}).
\end{displaymath}
Clearly $f_* \geq \rho$ in $\mathbb S^{n-1}$ and $\int_{\mathbb{S}^{n-1}} f_*(\omega)^n \dee o_\omega = \gamma$, while
\[
    \|\nabla_T f_*\|_{L^\infty(\mathbb{S}^{n-1})} \leq \liminf_{k \to \infty} \|\nabla_T f_k\|_{L^\infty(\mathbb{S}^{n-1})} \leq L,
\]
therefore $f_* \in \mathcal F$.
Let us write $\Omega_k=\Omega_{f_k}$ and $\Omega_*=\Omega_{f_*}$. We claim that $\Omega_k \rightarrow \Omega_*$ in the Hausdorff complementary
metric, i.e. $d_{\complement \Omega_k} \rightarrow d_{\complement \Omega_*}$ in $C(\bar D)$, where $d_A$ denotes the distance function to the set $A$ and $\complement A$ is the complement of a set $A$.
In order to prove
the claim we fix $x \in \bar D$ and choose $z \in \complement \Omega_*$ such that $d_{\complement \Omega_*}(x)= |x -z|$. Then $R \geq |z| \geq f_*(\omega_z) \geq \rho$. For 
$z_k=\bigl( 1+ \rho^{-1} \Vert f_k - f_* \Vert_{L^\infty(\mathbb{S}^{n-1})} \bigr) z$
we have $\omega_{z_k}=\omega_z$ and
\begin{displaymath}
| z_k| = |z| + \frac{|z|}{\rho} \Vert f_k - f_* \Vert_{L^\infty(\mathbb{S}^{n-1})} \geq f_*(\omega_z) + \Vert f_k - f_* \Vert_{L^\infty(\mathbb{S}^{n-1})} \geq f_k(\omega_z) = f_k(\omega_{z_k}).
\end{displaymath}
Therefore, $z_k \in \complement \Omega_k$ so that 
\begin{multline*}
d_{\complement \Omega_k}(x) - d_{\complement \Omega_*}(x) \leq | x - z_k | - |x -z | \leq |z_k - z| = \\ = \frac{|z|}{\rho} \Vert f_k - f_* \Vert_{L^\infty(\mathbb{S}^{n-1})} \leq \frac{R}{\rho} 
\Vert f_k - f_* \Vert_{L^\infty(\mathbb{S}^{n-1})},
\end{multline*}
where $R$ defined in \eqref{eq:DefnOfR} is the radius of the hold-all domain $D= B_R(0)$.
By exchanging the roles of $f_k$ and $f_*$ and taking the maximum with respect to $x$ we obtain 
\begin{displaymath}
\max_{x \in \bar D} | d_{\complement \Omega_k}(x) - d_{\complement \Omega_*}(x) |  \leq \frac{R}{\rho} 
\Vert f_k - f_* \Vert_{L^\infty(\mathbb{S}^{n-1})} \rightarrow 0, \, k \rightarrow \infty,
\end{displaymath}
which shows that $\Omega_k \to \Omega_*$ in the Hausdorff complementary metric.
Furthermore, according to \cite[Lemma 3, Section 3.2]{Bur98} the set $\Omega_*$ satisfies the cone condition and hence is locally Lipschitz. We may therefore deduce
from Theorem 4.1 in Chapter 8 of \cite{DelZol11} that $u_{\Omega_k} \rightarrow u_{\Omega_*}$ in $H^1_0(D)$. As a result
$J(f_*)= \lim_{k \rightarrow \infty} J(f_k) = \inf_{f \in \mathcal F} J(f)$ which completes the proof.
\end{proof}

\subsection{Calculating the shape derivative}
We now fix
\begin{equation}\label{eq:assumptionsOnf}
f \in W^{1,\infty}(\mathbb S^{n-1}) \mbox{ with } \min_{\omega \in \mathbb S^{n-1}} f(\omega)>0.
\end{equation}
In addition, fix $F \in L^2_{loc}(\R^{n})$, $z \in H^1_{loc}(\R^n)$.
These are defined on all of $\R^n$, rather than on a hold-all domain whose size depends on $\|\nabla_T f\|_{L^\infty(\mathbb S^{n-1})}$ and $\min_{\omega \in \mathbb S^{n-1}} f(\omega)$.
Before we calculate a formula for the directional derivative of $J$ at $f$ we transform the state equation to the reference domain $B$.
To do so, define $\hat u(x):= u(\Phi_f(x))$, where $u \in H^1_0(\Omega_f)$ denotes the solution of \eqref{statef} and $\Phi_f$ is given by
\eqref{defPhi}. Clearly, $\nabla u(\Phi_f(x)) = D \Phi_f(x)^{-t} \,  \nabla \hat u(x)$, where 
we think of the gradient as a column vector. Therefore, \eqref{statef} translates into
\begin{equation} \label{stateref}
\int_B A_f(\omega_x)   \nabla \hat u(x) \cdot \nabla \hat \eta(x) \, \dee x = \int_B \hat F_f(x) \hat \eta(x) \, f(\omega_x)^n \, \dee x \quad \forall \,  \hat \eta \in H^1_0(B).
\end{equation}
In the above $\hat F_f(x)=F(\Phi_f(x))$ and $A_f(\omega_x)= f(\omega_x)^n \, D \Phi_f(x)^{-1} D \Phi_f(x)^{-t}$. Using the fact that
\begin{displaymath}
D \Phi_f(x)^{-1} = \frac{1}{f(\omega_x)}  \Bigl( I - \omega_x \otimes \frac{\nabla_T f(\omega_x)}{ f( \omega_x)} \Bigr)
\end{displaymath}
we find that
\begin{equation} \label{af}
A_f(\omega_x) = f(\omega_x)^{n-2} \Bigl( I - \omega_x \otimes \frac{\nabla_T f( \omega_x)}{ f(\omega_x)}
- \frac{\nabla_T f (\omega_x)}{ f(\omega_x)} \otimes  \omega_x + \frac{| \nabla_T f(\omega_x) |^2}{f(\omega_x)^2}
\, \omega_x \otimes \omega_x \Bigr).
\end{equation}

\noindent
We wish to show that $J$ has a Gateaux derivative at $f$ in a direction $g \in W^{1,\infty}(\mathbb S^{n-1})$.
We define the vector--field
$V \in C^{0,1}(\mathbb R^n ;\mathbb R^n)$ by 
\begin{equation} \label{defV}
V(y)= \left\{
\begin{array}{cl}
\displaystyle \frac{g(\omega_y)}{f(\omega_y)} y, & y \neq 0, \\[2mm]
0, & y=0.
\end{array}
\right.
\end{equation}
Then, $(\mbox{id}+t V)(\Omega_f)= (\mbox{id}+ t V) \circ \Phi_f (B)$. For every $x \in B$ we have
\begin{displaymath}
\Phi_f(x)+t V(\Phi_f(x)) = \left\{
\begin{array}{cl}
\displaystyle (f(\omega_x)+t g(\omega_x)) x, & x \neq 0, \\[2mm]
0, & x=0,
\end{array}
\right.
\end{displaymath}
so that $(\mbox{id}+t V)(\Omega_f)= \Omega_{f+tg}$. \\[2mm]
We therefore have for $t \neq 0$ that
\begin{equation}\label{jstrich}
\frac{J(f+tg)-J(f)}{t} = \frac{\mathcal J((\mbox{id}+tV)(\Omega_f)) - \mathcal J(\Omega_f)}{t}.
\end{equation}
Since $\mathcal J$ is shape differentiable \cite[Proposition 4.4]{ADJ21}, the right hand side of the above converges to $\mathcal{J}'(\Omega_f)(V)$ as $t \to 0$ with $J'(\Omega_f)(V)$ representing a linear mapping, see \eqref{eq:ShapeDerivativeGeneralBulkForm} and \eqref{eq:BoundaryDerivativeGeneralCase}.
Therefore we see $J$ is Gateaux differentiable.

By  adapting the proof of \cite[Proposition 4.5]{ADJ21} to our situation
we obtain the {\it volume form} of the shape derivative as
\begin{eqnarray}
\mathcal J'(\Omega_f)(V) & = &    \int_{\Omega_f} \bigl( DV+ DV^t - \mbox{div}V \, I  \bigr) \nabla u \cdot \nabla p \, \dee x   \nonumber  \\
&&  +  \int_{\Omega_f} \bigl( \frac{1}{2} (u-z)^2 \, \mbox{div} V -(u-z) \nabla z \cdot V \bigr) \, \dee x - \int_{\Omega_f} 
  F V \cdot \nabla p  \, \dee x. \label{eq:ShapeDerivativeGeneralBulkForm} 
\end{eqnarray}
Here $p \in H^1_0(\Omega_f)$ is  the solution of
the adjoint problem
\begin{equation} \label{adjoint}
\int_{\Omega_f} \nabla p \cdot \nabla \eta \, \dee x = \int_{\Omega_f} (u-z) \eta \, \dee x \qquad \forall \eta \in H^1_0(\Omega_f).
\end{equation}


If in addition, $u,\,p \in H^2(\Omega_f)$, the shape derivative can be written in the well-known boundary form
\begin{equation}\label{eq:BoundaryDerivativeGeneralCase}
    \mathcal J '(\Omega_f)(V) = \int_{\partial \Omega_f} \left( \frac{1}{2} (u-z)^2 + \frac{\partial u}{\partial \nu} \frac{\partial p}{\partial \nu} \right) V\cdot \nu \, \dee S,
\end{equation}
where $\nu$ is almost everywhere the outward unit normal to $\Omega_f$ and $\dee S$ denotes the surface element on $\partial \Omega_f$, see \cite[Theorem 4.6]{ADJ21}.
\subsubsection{ Mapping the volume form  \eqref{eq:ShapeDerivativeGeneralBulkForm}  to the reference domain $B$.} By changing variable in \eqref{eq:ShapeDerivativeGeneralBulkForm} we have
\begin{eqnarray}
\lefteqn{ \langle J'(f),g \rangle  =  \int_B (D \Phi_f)^{-1} \bigl( DV+ DV^t - (\mbox{div}V) \, I  \bigr) \circ \Phi_f (D \Phi_f)^{-t} \nabla \hat u \cdot \nabla \hat p \;
f(\omega_x)^n \, \dee x }  \nonumber  \\
& & + \int_B \bigl( \frac{1}{2} (\hat u - \hat z_f)^2 \, (\mbox{div}V) \circ \Phi_f - (\hat u - \hat z_f) \nabla \hat z_f \cdot (D \Phi_f)^{-1} V \circ \Phi_f \bigr) f(\omega_x)^n \, \dee x
\nonumber \\
  & &  - \int_B \hat F_f (D \Phi_f)^{-1} V \circ \Phi_f \cdot \nabla \hat p \,  f(\omega_x)^n \, \dee x, \label{shapederivref}
\end{eqnarray}
where $\hat z_f(x)=z(\Phi_f(x))$ and we have used \eqref{jstrich}.
In the same way as above we obtain from \eqref{adjoint} that $\hat p(x)=p(\Phi_f(x))$ satisfies
\begin{equation} \label{adjointref}
\int_B A_f(\omega_x)   \nabla \hat p(x) \cdot \nabla \hat \eta(x) \, \dee x = \int_B (\hat u(x) - \hat z_f(x))  \hat \eta(x) \, f(\omega_x)^n \, \dee x \; \;  \forall \, \hat \eta \in H^1_0(B).
\end{equation}
Differentiating the relation $V(\Phi_f(x))= g(\omega_x)x$ 
we obtain
\[
    DV(\Phi_f(x)) D \Phi_f(x)= g(\omega_x)I + \omega_x \otimes \nabla_T g(\omega_x)
\]
and hence
\begin{displaymath}
DV \circ  \Phi_f  =  \bigl( g I + \omega_x \otimes \nabla_T g \bigr) (D \Phi_f)^{-1}  =  \frac{1}{f} \bigl( g I - \frac{g}{f} \, \omega_x \otimes \nabla_T f + \omega_x \otimes \nabla_T g \bigr).
\end{displaymath}
In particular we deduce that
\begin{displaymath}
({\Div} V) \circ \Phi_f = \mbox{trace}\, DV \circ \Phi_f = n \frac{g}{f}
\end{displaymath}
as well as
\begin{eqnarray*}
\lefteqn{ DV \circ \Phi_f + DV^t \circ \Phi_f - {\Div} V \circ \Phi_f \, I} \\
& = & \frac{g}{f} (2-n) I - \frac{g}{f^2} \omega_x \otimes \nabla_{T} f - \frac{g}{f^2} \nabla_T f \otimes \omega_x + \frac{1}{f} \omega_x \otimes \nabla_T g
+ \frac{1}{f} \nabla_T g \otimes \omega_x.
\end{eqnarray*}
A long, but straightforward calculation then shows that
\begin{eqnarray*}
\lefteqn{ (D \Phi_f)^{-1} \bigl(  DV \circ \Phi_f + DV^t \circ \Phi_f - \mbox{div} V \circ \Phi_f \, I \bigr) (D \Phi_f)^{-t} } \\
& = & \frac{g}{f^3} (2-n) I + (n-3) \frac{g}{f^4} \bigl( \omega_x \otimes \nabla_T f + \nabla_T f \otimes \omega_x \bigr) +
\frac{1}{f^3} \bigl( \omega_x \otimes \nabla_T g + \nabla_T g \otimes \omega_x \bigr) \\
& & + \bigl( (4-n) \frac{g}{f^5} | \nabla_T f |^2 -2 \frac{1}{f^4} \bigl( \nabla_T f \cdot \nabla_T g \bigr) \bigr) \omega_x \otimes \omega_x.
\end{eqnarray*}
Note also that
\begin{displaymath}
(D \Phi_f)^{-1} V \circ \Phi_f = \frac{1}{f} \bigl( I - \omega_x \otimes \frac{\nabla_T f}{f} \bigr) g x = \frac{g}{f} x = |x| \frac{g}{f} \omega_x.
\end{displaymath}
If we insert the above identities into \eqref{shapederivref} and transform to polar coordinates we obtain 
\begin{equation} \label{shaped}
\langle J'(f),g \rangle = \int_B \bigl( h_f g + H_f \cdot \nabla_T g) \dee x =   \int_{\mathbb S^{n-1}} \bigl( \tilde h_f g + \tilde H_f \cdot \nabla_T g \bigr) do_\omega,
\end{equation} 
where $h_f\colon  B \rightarrow \mathbb R$ and $H_f\colon B  \rightarrow \mathbb R^n$ are defined by
\begin{eqnarray}
\; \; \; h_f & = &  (2-n) f^{n-3} \nabla \hat u \cdot \nabla \hat p  + (4-n) f^{n-5} | \nabla_T f|^2 (\nabla \hat u \cdot \omega_x) (\nabla \hat p \cdot \omega_x)  \label{hf} \\
& & + (n-3) f^{n-4}  \bigl( (\nabla_T f \cdot \nabla \hat u ) 
 (\omega_x \cdot \nabla \hat p) + (\nabla_T f \cdot \nabla \hat p)
( \omega_x \cdot \nabla \hat u) \bigr) \nonumber \\
& &  + f^{n-1}  \bigl( \frac{n}{2} (\hat u - \hat z_f)^2 - |x|  (\hat u - \hat z_f) \nabla \hat z_f \cdot \omega_x - |x|  \hat F_f \, \nabla \hat p
\cdot \omega_x \bigr); \nonumber \\
\; \; \; H_f & = & f^{n-3} \bigl( (\nabla \hat p \cdot \omega_x) \nabla \hat u + (\nabla \hat u  \cdot \omega_x) \nabla \hat p \bigr)   
 - 2 f^{n-4}  (\nabla \hat u \cdot \omega_x) (\nabla \hat p \cdot \omega_x) \,   \nabla_T f, \label{Hf} 
\end{eqnarray}
while $\tilde h_f\colon \mathbb S^{n-1} \rightarrow \mathbb R, \, \tilde H_f\colon \mathbb S^{n-1} \rightarrow \mathbb R^n$ are given by
\begin{displaymath}
\tilde h_f(\omega)= \int_0^1 s^{n-1} h_f(s \omega) \dee s, \; \tilde H_f(\omega) = \int_0^1 s^{n-1} H_f(s \omega) \dee s.
\end{displaymath}
From our assumptions on $z$ and $f$ we deduce that $\tilde h_f \in L^1(\mathbb S^{n-1})$, $\tilde H_f \in L^1(\mathbb S^{n-1};\mathbb R^n)$. \\

\noindent
\subsubsection{ Mapping the boundary form of the shape derivative to $\mathbb{S}^{n-1}$.}
As we intend to use  formula \eqref{eq:BoundaryDerivativeGeneralCase} also for numerical purposes we
transform it to an integral over the reference boundary $\mathbb S^{n-1}$ with the help of the mapping 
\begin{displaymath}
\Phi_{f| \mathbb S^{n-1}}\colon \mathbb S^{n-1} \rightarrow \partial \Omega_f, \; \omega \mapsto f(\omega) \omega.
\end{displaymath}
A calculation of $\dee S$, the surface element on $\partial \Omega_f$, shows
\begin{equation} \label{transdet}
\dee S = f(\omega)^{n-1} \left(1 + \frac{| \nabla_{T} f(\omega) |^2}{f(\omega)^2}\right)^{1/2} \dee o_\omega,
\end{equation}
while
\begin{displaymath}
(\nu \circ \Phi_f)(\omega)= \frac{(D \Phi_f(\omega))^{-t} \omega}{| (D \Phi_f(\omega))^{-t} \omega |} = \left(1+  \frac{| \nabla_{T} f(\omega) |^2}{f(\omega)^2}\right)^{-\frac{1}{2}}
\bigl( \omega - \frac{\nabla_T f(\omega)}{f(\omega)} \bigr).
\end{displaymath}
Since $(\nabla u \circ \Phi_f)(\omega) = (D \Phi_f(\omega))^{-t} \nabla \hat u(\omega)$ we deduce that
\begin{eqnarray*}
\frac{\partial u}{\partial \nu} \circ \Phi_f & = & \left(1+  \frac{| \nabla_{T} f |^2}{f^2}\right)^{-\frac{1}{2}} (D \Phi_f)^{-t} \nabla \hat u \cdot \bigl( \omega - \frac{\nabla_T f}{f} \bigr) \\
& = & \frac{1}{f} \left(1+  \frac{| \nabla_{T} f |^2}{f^2}\right)^{-\frac{1}{2}}
\bigl( I - \frac{\nabla_T f}{f} \otimes \omega \bigr) \nabla \hat u \cdot 
\bigl( \omega - \frac{\nabla_T f}{f} \bigr) \\
& = & \frac{1}{f} \left(1+  \frac{| \nabla_{T} f |^2}{f^2}\right)^{-\frac{1}{2}} (\omega - \frac{\nabla_T f}{f}) \cdot
( \nabla \hat{u} -  \frac{\nabla_T f}{f}\frac{\partial\hat{u}}{\partial \omega})\\
& = & \frac{1}{f} \left(1+  \frac{| \nabla_{T} f |^2}{f^2}\right)^{\frac{1}{2}} \frac{\partial \hat u}{\partial \omega},
\end{eqnarray*}
where we have used that $\nabla \hat{u} \cdot \nabla_T f = 0$ on $\partial B$ since $\hat{u} =0 $ on $\partial B$.
For the function $V$ given by \eqref{defV} we have $(V \circ \Phi_f)(\omega)=g(\omega) \omega$ and hence by \eqref{nablat}
\begin{displaymath}
(V \cdot \nu) \circ \Phi_f = \left(1+  \frac{| \nabla_{T} f |^2}{f^2}\right)^{-\frac{1}{2}} \, g.
\end{displaymath}
After a change of variables in \eqref{eq:BoundaryDerivativeGeneralCase}, using \eqref{transdet} as well as the formulae above we find
\begin{align*}
\langle J'(f),g \rangle  &= \int_{\mathbb S^{n-1}} \left( \frac{1}{2} (\hat u - \hat z_f)^2 + (\frac{\partial u}{\partial \nu}  \frac{\partial p}{\partial \nu})
\circ \Phi_f \right) (V \cdot \nu) \circ \Phi_f  f^{n-1} \left(1 + \frac{| \nabla_{T} f |^2}{f^2}\right)^{\frac{1}{2}} \dee o_\omega  \\
& =  \int_{\mathbb S^{n-1}} \boundaryFormOfDerivative g \dee o_\omega,
\end{align*}
where $ \boundaryFormOfDerivative \colon \mathbb S^{n-1} \rightarrow \mathbb R$ is given by 
\begin{equation} \label{hfboundary}
\boundaryFormOfDerivative = \frac{1}{2} (\hat u(\omega) - \hat z_f(\omega))^2 f(\omega)^{n-1}  + f(\omega)^{n-3} \bigl( 1 + \frac{| \nabla_T f(\omega)|^2}{f(\omega)^2} \bigr)
\frac{\partial \hat u}{\partial \omega}(\omega) \frac{\partial \hat p}{\partial \omega}(\omega).
\end{equation}

\subsubsection{ A descent direction in the  $W^{1,\infty}$-topology.}
We wish to consider perturbations which preserve the volume constraint $ \int_{\mathbb S^{n-1}} f(\omega)^n \dee o_\omega= \gamma$ to first order; we therefore introduce the following set of admissible perturbations
\begin{eqnarray*}
        V_\infty(f)& := & \left\{ v \in W^{1,\infty}(\mathbb S^{n-1}) \,  : 
       \int_{\mathbb S^{n-1}} f^{n-1} v \, \dee o_\omega =0, \; \|\nabla_T v\|_{L^\infty(\mathbb S^{n-1})} \leq 1 \right\}.
\end{eqnarray*}

Before we show that there is a minimising direction in $V_\infty(f)$, we have the following Lemma which shows that the Lipschitz semi-norm is equivalent to the Lipschitz norm on $V_\infty(f)$.
\begin{lemma}\label{lem:EquivalentNorms}
    Let $f$ be as in \eqref{eq:assumptionsOnf}.
    Then we have for all $v \in W^{1,\infty}(\mathbb S^{n-1})$ with $\int_{\mathbb S^{n-1}} f^{n-1} v \dee o_\omega = 0$, 
    \[
        \|v \|_{L^\infty( \mathbb{S}^{n-1} )} \leq \pi \|\nabla_T v\|_{L^\infty( \mathbb{S}^{n-1} )} .
    \]
\end{lemma}
\begin{proof}
    Since $\int_{\mathbb{S}^{n-1}} f^{n-1}v\dee o_\omega = 0$ there exists  $\omega_v \in \mathbb{S}^{n-1}$ such that $v(\omega_v) = 0$, which combined with \eqref{eq:LipschitzBoundProof} implies that
\begin{displaymath}\begin{split}
\Vert v \Vert_{L^{\infty}(\mathbb S^{n-1})} &\leq \max_{\omega \in \mathbb S^{n-1}} | v(\omega) - v(\omega_v) |  \leq \frac{\pi}{2}  \|\nabla_T v\|_{L^\infty( \mathbb{S}^{n-1} )}  \max_{\omega \in 
\mathbb S^{n-1}} | \omega - \omega_v |
\\&\leq \pi \|\nabla_T v\|_{L^\infty( \mathbb{S}^{n-1} )}  .
\end{split}
\end{displaymath}
\end{proof}

\begin{theorem} \label{optimaldesc}
    Let $f$ be as in \eqref{eq:assumptionsOnf}.
Then there exists $g \in V_\infty(f)$ such that $\langle J'(f),g \rangle=\min_{v \in V_\infty(f)} \langle J'(f),v \rangle$.
\end{theorem}
\begin{proof}
We adapt the proof of Proposition 3.1 in \cite{PaWeFa18} to our setting.
Let $(v_k)_{k \in \mathbb N} \subset V_\infty(f)$ be a sequence such that $\langle J'(f),v_k \rangle \searrow \inf_{v \in V_\infty(f)} \langle J'(f),v \rangle$.
Applying Lemma \ref{lem:EquivalentNorms} gives $\|v_k\|_{W^{1,\infty}(\mathbb{S}^{n-1})} \leq \pi +1$, therefore, as in the proof of Theorem \ref{thm:ExistsConditionalMinimiser}, one has that there is a sequence $k_j \to \infty$ as $j \to \infty$
and $g \in W^{1,\infty}(\mathbb S^{n-1})$ such that 
\begin{displaymath}
v_{k_j} \rightarrow g \mbox{ in } C(\mathbb S^{n-1}) \mbox{ and } \nabla_T v_{k_j} \overset{*}{\rightharpoonup} \nabla_T g \mbox{ in }L^\infty(\mathbb{S}^{n-1}).
\end{displaymath}
In particular we have that $g \in V_\infty(f)$.
Furthermore, since $\tilde h_f \in L^1(\mathbb S^{n-1})$ and $\tilde H_f \in L^1(\mathbb S^{n-1}; \mathbb R^n)$,
\begin{eqnarray*}
\langle J'(f), g \rangle & = &  \int_{\mathbb S^{n-1}} \bigl( \tilde h_f g + \tilde  H_f \cdot \nabla_T g \bigr) \dee o_\omega
= \lim_{j \rightarrow \infty} \int_{\mathbb S^{n-1}} \bigl( \tilde h_f v_{k_j} + \tilde  H_f \cdot \nabla_T v_{k_j} \bigr) \dee o_\omega \\
& = &  \lim_{j \rightarrow \infty} \langle J'(f), v_{k_j} \rangle = \inf_{v \in V_\infty(f)} \langle J'(f),v \rangle.
\end{eqnarray*}
\end{proof}

In the case where $f$ is regular enough, so that the boundary form of the derivative \eqref{eq:BoundaryDerivativeGeneralCase} exists, in particular when $\boundaryFormOfDerivative\in L^1(\mathbb S^{n-1})$, one may show the existence of an optimal descent direction for the boundary form of the derivative with the same argument as above.
\begin{remark} \label{rem:IL05}
We would like to establish a connection between Theorem \ref{optimaldesc} and a theory developed by Ishii and Loreti
in \cite{IshLor05}. To do so, let us assume that $\tilde h_f \in L^\infty(\mathbb S^{n-1})$ and $\tilde H_f \in W^{1,\infty}(\mathbb S^{n-1};\mathbb R^n)$.
Then we obtain after integration by parts on $\mathbb{S}^{n-1}$ that
\begin{displaymath}
\langle J'(f),v \rangle = \int_{\mathbb S^{n-1}} \bigl( \tilde h_f - \nabla_T \cdot \tilde H_f +(n-1) \tilde H_f \cdot \omega - c f^{n-1} \bigr) v \dee o_\omega, \quad v \in V_\infty(f)
\end{displaymath}
where $\nabla_T \cdot \tilde H_f = \sum_{i=1}^n e_i \cdot \nabla_T (\tilde H_f)_i$ and
\begin{displaymath}
c= \bigl( \int_{\mathbb S^{n-1}} f^{n-1} \dee o_\omega \bigr)^{-1}  \int_{\mathbb S^{n-1}} \bigl( \tilde h_f - \nabla_T \cdot \tilde H_f +(n-1) \tilde H_f \cdot \omega 
\bigr) \dee o_\omega.
\end{displaymath}
We note that the $(n-1) \tilde{H}_f \cdot \omega$ term arises from the mean curvature of $\mathbb{S}^{n-1}$, see \cite[Equation (2.16)]{DecDziEll05} for example.
If we let
\begin{equation}\label{eq:definitionOfQ}
    q_f:= \tilde h_f - \nabla_T \cdot \tilde H_f +(n-1) \tilde H_f \cdot \omega - c f^{n-1}
\end{equation}
we have $q_f \in L^\infty(\mathbb S^{n-1})$ and
\begin{displaymath}
\langle J'(f),v \rangle = \int_{\mathbb S^{n-1}} q_f v \dee o_\omega, \quad v \in V_\infty(f), \quad \mbox{ with } \int_{\mathbb S^{n-1}} q_f \dee o_\omega=0.
\end{displaymath}
Let us consider for $p \geq 2$ 
\[
    g_p := \argmin_{\{ v \in W^{1,p}(\mathbb S^{n-1})\,|\, \int_{\mathbb S^{n-1}} 
f^{n-1} v \dee o_\omega=0 \} }\left\{ \frac{1}{p} \int_{\mathbb S^{n-1}} | \nabla_T v |^p \dee o_\omega + \int_{\mathbb S^{n-1}} q_f v \dee o_\omega \right\}.
\]
By adapting the arguments of \cite[Section 5]{IshLor05} to our setting it is then possible to show that there exists a sequence  $p_j \to \infty$ as $j \to \infty$
and $g \in V_\infty(f)$ such that  $(g_{p_j})_{j \in \mathbb N}$ converges to $g$  uniformly on $\mathbb S^{n-1}$ and $\langle J'(f),g \rangle=\min_{v \in V_\infty(f)} \langle J'(f),v \rangle$.
As a result, the optimal descent direction in
Theorem 2.4 can be approximated with the help of the solution of a $p$--Laplace problem. This relationship has been exploited with promising results for a fluid dynamics application
in \cite{MulKulSie21} for domains 
that are not necessarily star--shaped. \\
In this context we note that it is also possible to consider a minimiser over H\"older functions, rather than Lipschitz functions, which is done in \cite{Jyl15}.
\end{remark}

\subsection{Relation to Optimal Transport} \label{sec:OT}
Let us again assume as in  Remark \ref{rem:IL05} that $\tilde h_f \in L^\infty(\mathbb S^{n-1})$ and $\tilde H_f \in W^{1,\infty}(\mathbb S^{n-1};\mathbb R^n)$ and hence
\begin{displaymath}
\langle J'(f),v \rangle = \int_{\mathbb S^{n-1}} q_f v \dee o_\omega, \quad v \in V_\infty(f),
\end{displaymath}
where $q_f$ is defined in \eqref{eq:definitionOfQ} and satisfies $q_f \in L^\infty(\mathbb S^{n-1})$ and $\int_{\mathbb S^{n-1}} q_f \dee o_\omega =0$. 
Abbreviating $q_f^+:=\max(q_f,0), q_f^-=-\min(q_f,0)$ we have
$q_f=q_f^+- q_f^-$ as well as $\int_{\mathbb S^{n-1}} q_f^+ \dee o_\omega = \int_{\mathbb S^{n-1}} q_f^- \dee o_\omega$. In what follows we assume for simplicity that
both integrals are equal to 1. Then, $\mu^\pm=q_f^\pm \dee o_\omega$ are probability measures and
\begin{displaymath}
\langle J'(f), v \rangle = \int_{\mathbb S^{n-1}} v \, \dee (\mu^+ - \mu^-), \quad v \in V^\infty(f). 
\end{displaymath}
In order to see the link to optimal transport, it is convenient to consider a maximisation problem instead of a minimisation problem.
Since $v \mapsto \langle J'(f), v \rangle$ is linear, we may convert our existing minimisation problem, which appears in Theorem \ref{optimaldesc} to a maximisation problem by a modification of signs.
Observe that
\begin{equation} \label{max}
\max_{v \in V_\infty (f)} \langle J'(f),v \rangle = \max_{v \in W^{1,\infty}(\mathbb S^{n-1}), \Vert \nabla_T v \Vert_{L^\infty(\mathbb S^{n-1})} \leq 1} \int_{\mathbb S^{n-1}} q_f v \dee o_\omega
\end{equation}
where we have used that $\int_{\mathbb S^{n-1}}q_f \dee o_\omega =0$. Furthermore, it is possible to verify that 
\begin{displaymath}
\lbrace v \in W^{1,\infty}(\mathbb S^{n-1}) \, | \, \Vert \nabla_T v \Vert_{L^\infty(\mathbb S^{n-1})} \leq 1 \rbrace \cong \mbox{Lip}_1,
\end{displaymath}
where $\mbox{Lip}_1$ denotes the set of all $v \in C^{0,1}(\mathbb S^{n-1})$ with ${|v(x)-v(y)|}\leq {d(x,y)}$ for all $x,\,y \in \mathbb{S}^{n-1}$.
As a result we find with the help of \eqref{max} and the duality relation which appears in Equation (3.1) of \cite{San15} that
\begin{eqnarray}
\max_{v \in V_\infty(f)} \langle J'(f),v \rangle & = &  \max\left\{ \int_{\mathbb{S}^{n-1}} v \,\dee (\mu^+ - \mu^-) \, | \, v \in \mbox{Lip}_1 \right\} \nonumber \\
& = & \min \left\{ \int_{\mathbb{S}^{n-1} \times \mathbb{S}^{n-1}} d(x,y) \,\dee \psi(x,y) \,|\, \psi \in \Pi( \mu^+,\mu^-) \right\}. \label{maxmin}
\end{eqnarray}
In the above,
\begin{equation*}
    \Pi(\mu^+,\mu^-) := \{ \psi \in \mathcal{P}(\mathbb{S}^{n-1}\times \mathbb{S}^{n-1}) \,|\, (\pi_x)_\# \psi = \mu^+,\, (\pi_y)_\#\psi = \mu^- \},
\end{equation*}
where $\mathcal{P}$ denotes probability measures and $\pi_x$ and $\pi_y$ are the projections from $\mathbb{S}^{n-1}\times \mathbb{S}^{n-1}$ onto the first and second components respectively.
Thus, the problem in Theorem 2.4 (with $\min$ replaced by $\max$) is the dual of the optimal transport problem of minimising the cost of transporting $\mu^+$ to $\mu^-$ with the
cost given by the spherical distance.
This relation to Optimal Transport will be exploited in Section \ref{sec:LipOTImplementation} as a method to produce an approximation of a direction of steepest descent.

\subsection{Steepest descent for $n=2$}\label{sec:SteepestDescentNIs2}
The determination of the minimiser $g$ in Theorem \ref{optimaldesc}  is by no means straightforward. In what follows we shall focus on
the case $n=2$ and write $\bar f(\phi)=f( (\cos(\phi), \sin(\phi) )^t)$ for $f \in W^{1,\infty}(\mathbb S^1)$ and $\phi \in [0,2\pi]$. Since
$ \bar {f}'(\phi) (-\sin(\phi), \cos(\phi))^t=\nabla_{T} f((\cos(\phi), \sin(\phi))^t)$ we obtain the following form of \eqref{shaped}:
\begin{equation} \label{sdint}
\langle J'(f),v \rangle = \int_0^{2 \pi} \bigl( \bar h_f(\phi) v( \phi) + \bar H_f(\phi)  v'(\phi) \bigr) \, \dee \phi, \quad v \in W^{1,\infty}_{\mbox{per}}(0,2 \pi).
\end{equation}
Here, $\bar h_f(\phi)=\tilde h_f((\cos(\phi), \sin(\phi))^t)$, $\bar H_f(\phi)= \tilde H_f((\cos(\phi), \sin(\phi))^t)\cdot (-\sin(\phi), \cos(\phi))^t$ for $\phi \in [0, 2 \pi]$. The boundary form of the shape
derivative, when sufficiently regular, can be treated analogously where one replaces occurrences of $\tilde h_f$ with $\boundaryFormOfDerivative$ and replaces $\tilde H_f$ with $0$. \\
In this setting the set of admissible directions becomes
\begin{displaymath}
V_\infty(f)= \left\lbrace v \in W^{1,\infty}_{\mbox{per}}(0,2 \pi) \, | \, \int_0^{2 \pi} \bar f  v \dee \phi =0, \;
\Vert v' \Vert_{L^{\infty}(0,2\pi)} \leq 1 \right\rbrace.
\end{displaymath}
In order to proceed and motivate our numerical approach we assume  the situation of Remark \ref{rem:IL05}, namely 
that $\bar h_f \in L^\infty(0,2\pi)$ and $\bar H_f \in W^{1,\infty}_{\mbox{per}}(0,2 \pi)$.
Then
after integration by parts and using the condition $\int_0^{2 \pi} \bar f \, v \dee \phi=0$ for $v \in V^\infty(f)$, we obtain
\begin{displaymath}
\langle J'(f), v \rangle = \int_0^{2 \pi} (\bar h_f(\phi) - \bar H_f'(\phi) - c \bar f(\phi)) v(\phi) \dee \phi, \quad \mbox{ where }  
c=\frac{\int_0^{2 \pi} \bar h_f(\phi) \dee \phi}{\int_0^{2 \pi} \bar f(\phi) \dee \phi}.
\end{displaymath}
If we let $q_f:= \bar h_f - \bar H_f' - c \bar f \in L^\infty(0,2\pi)$ we have
\begin{displaymath}
\langle J'(f),v \rangle = \int_0^{2 \pi} q_f(\phi) v(\phi) \dee \phi, \quad v \in V_\infty(f),
\end{displaymath}
as well as
\begin{displaymath}
\int_0^{2 \pi} q_f(\phi) \dee \phi= \int_0^{2 \pi} \bar h_f(\phi) \dee \phi - c \int_0^{2 \pi} \bar f(\phi) \dee \phi - \bar  H_f(2 \pi) + \bar H_f(0) =0
\end{displaymath}
by the choice of $c$.
Our aim is to obtain a function $\bar{g} \in V_\infty(f)$ such that
\begin{equation} \label{maxbc}
\int_0^{2 \pi} q_f(\phi) \bar g(\phi) \dee \phi= \min_{v \in V_\infty(f)} \int_0^{2 \pi} q_f(\phi) v(\phi) \dee \phi.
\end{equation}
To do so, we introduce 
\begin{equation} \label{defG}
G(\phi) :=  -\int_0^\phi q_f(t) \dee t
\end{equation}
and note that $G(2\pi) = G(0) = 0$.
We have for any $v \in V_\infty(f)$ and any $\beta \in \R$
\begin{multline}\label{eq:EquationForKlausMethod}
    \int_0^{2\pi} q_f(\phi)v(\phi)\dee \phi
    = - \int_0^{2\pi} G'(\phi) v(\phi) \dee \phi
    = \int_0^{2\pi} G(\phi)v'(\phi)\dee \phi
    \\
    = \int_0^{2\pi} (G (\phi) - \beta) v'(\phi)\dee \phi
    \geq - \int_0^{2\pi} |G(\phi)-\beta| \dee \phi.
\end{multline}
In order to proceed we define as in \cite[p 414]{IshLor05}
\begin{eqnarray}
\; M(r) &:= & |\{\phi \in [0,2\pi) : G(\phi) < r \}|, r \in \mathbb R; \: \beta^* := \sup \{ r \in \mathbb{R} : M(r) \leq \pi\}; \label{Mbeta} \\
O_{\pm} &:=& \{ \phi\in [0,2\pi) : G(\phi) \gtrless \beta^*\},  \,
    O_0 := \{ \phi \in[0,2\pi) : G(\phi) = \beta^*\}; \label{defO} \\
k & := & \left\{ 
\begin{array}{cl}
0, & |O_0| = 0,\\
\frac{|O_+|-|O_-|}{|O_o|}, & \mbox{otherwise}.
\end{array}
\right. \label{defk}
\end{eqnarray}
From the choice of $\beta^*$, we deduce from Lemma 3.5 in \cite{IshLor05} that $|k|\leq 1$ and set
\begin{equation}\label{defg}
    \bar{g}(\phi) = \int_0^\phi \left(\chi_{O_-}(t) - \chi_{O_+}(t) + k \chi_{O_0}(t) \right)\dee t + \alpha,
\end{equation}
where $\alpha \in \R$ is chosen in such a way that $\int_0^{2\pi} \bar{g}(\phi) \bar f(\phi) \dee \phi = 0$.


\begin{proposition}\label{prop:minimiser}
The function $\bar{g}$ defined in \eqref{defg} belongs to $V_\infty(f)$ and satisfies
\[
    \langle J'(f),\bar{g} \rangle = \min_{v \in V_\infty(f)} \langle J'(f),v\rangle.
\]
\end{proposition}
\begin{proof}
    By the choice of $\beta^*$ we see that $\int_0^{2\pi} \left(\chi_{O_-}(t) - \chi_{O_+}(t) + k \chi_{O_0}(t) \right)\dee t = |O_-|-|O_+|+k|O_0| = 0$.
    Moreover, $\bar{g} \in W^{1,\infty}_{per}(0,2\pi)$ with $\|\bar g'\|_{L^\infty(0,2\pi)} \leq \max(1,|k|) =1$, so that $\bar{g} \in V_\infty(f)$.
    Furthermore, for every $v \in V_\infty(f)$ and setting $\beta = \beta^* $ in \eqref{eq:EquationForKlausMethod},
    \begin{align*}
        \int_0^{2\pi} q_f(\phi) v(\phi) \dee \phi
        \geq&
        -\int_0^{2\pi} |G(\phi) - \beta^*| \dee \phi
        \\
        =&
        \int_{O_-} (G(\phi) - \beta^*) \dee \phi - \int_{O_+}(G(\phi) - \beta^*) \dee \phi
        \\
        =&
        \int_0^{2\pi} (G(\phi) - \beta^*) \bar{g}'(\phi) \dee \phi
        =
        \int_{0}^{2\pi} q_f(\phi) \bar{g}(\phi)\dee \phi,
    \end{align*}
    where the first equality follows from the choice of $\beta^*$ in \eqref{Mbeta} and the definition of $O_\pm$ in \eqref{defO}, and the second equality follows from $\bar g' = \mp1$ in $O_\pm$.
    Recalling that $\langle J'(f), v \rangle = \int_0^{2\pi} q_f(\phi)v(\phi) \dee \phi$ for $v \in V_\infty(f)$, the result follows.
\end{proof}

\section{Discretisation}

\subsection{Approximation of the shape derivative}

We use the above ideas in order to set up  numerical schemes in two space dimensions. To do so, we approximate both the radial function $f$ and the solutions to the state
and adjoint equations with the help of continuous, piecewise linear finite elements, but on grids that are independent of each other. Let
$\mathcal{T}_h  $ be a quasi-uniform triangulation of (a subset of) the unit ball $B$, where $B_h := \left(\bigcup_{T \in \mathcal{T}_h} T \right)^\circ \subset B$ and 
the vertices on $\partial B_h$ lie on $\partial B$.
We define $\mathcal{S}_h$ to be
\[
    \mathcal{S}_h := \{ \hat \eta_h \in C(B_h) \, | \, \hat \eta_h=0 \mbox{ on } \partial B_h, \, \hat  \eta_{h|T} \in P^1(T), \, T \in \mathcal T_h\},
\]
where $P^1(T)$ denotes the set of polynomials on $T$ of degree 1 or less.
Next, given $N \in \mathbb{N}$, we set $\phi_i=2 \pi \frac{i}{N}, i=0,\ldots,N$ as well as
\[
    \mathcal{S}^N := \{ \bar v \in C([0, 2 \pi]) : \bar v_{| [\phi_{i-1},\phi_i]} \in P^1([\phi_{i-1},\phi_i]), i=1,\ldots,N, \bar v(0)=\bar v(2 \pi) \},
\]
the set of continuous, piecewise linear, periodic functions on $[0, 2 \pi]$. \\
Given $\bar f \in  \mathcal{S}^N$, we set $f(\omega)=\bar f(\phi)$ if $\omega = (\cos(\phi),\sin(\phi))^t$ and define $\hat{u}_h, \hat{p}_h \in \mathcal{S}_h$ as the unique solutions of
\begin{eqnarray} 
\int_{B_h} A_f(\omega_x)  \nabla \hat u_h \cdot \nabla \hat \eta_h \, \dee x = \int_{B_h} \hat F_f \hat \eta_h \, f(\omega_x)^2 \, \dee x; \quad \forall \hat \eta_h \in
\mathcal{S}_h, \label{staterefdisc} \\
\int_{B_h} A_f(\omega_x)  \nabla \hat p_h \cdot \nabla \hat \eta_h \, \dee x = \int_{B_h} (\hat u_h - \hat z_f)  \hat \eta_h \, f(\omega_x)^2 \, \dee x \quad \forall \hat \eta_h \in
\mathcal{S}_h, \label{adjointrefdisc}
\end{eqnarray}
where the integrals are approximated with a quadrature which is exact on polynomials of degree 6 on each of the triangles of $\mathcal{T}_h$. 
For convenience, we remind the reader that for $x \in B $, $x \neq 0$, ${A}_f(\omega_x) = f(\omega_x)^{2}D\Phi_f(x)^{-1}D\Phi_f(x)^{-t}$, where $\Phi_f(x) = f(\omega_x) x$ and $\omega_x = \frac{x}{|x|}$, we recall that $\hat{z}_f  = z \circ \Phi_f$ and $\hat{F}_f = F \circ \Phi_f$  on $B$.
Let us use $\hat{u}_h$ and $\hat{p}_h$ in order to define 
discrete versions of \eqref{hf},  \eqref{Hf} as well as \eqref{hfboundary}: \\
\subsubsection{Volume form of the shape derivative} Let $h_{f,h}\colon B_h \rightarrow \mathbb R, \, H_{f,h}\colon B_h \rightarrow \mathbb R^2$ be defined by
\begin{eqnarray}
\label{eq:hfh}
 h_{f,h} & = &  2 \frac{| \nabla_T f|^2}{f^3}  (\nabla \hat u_h \cdot \omega_x) (\nabla \hat p_h \cdot \omega_x) \\
& & - \frac{1}{f^2}  \bigl( (\nabla_T f \cdot \nabla \hat u_h ) 
 (\omega_x \cdot \nabla \hat p_h) + (\nabla_T f \cdot \nabla \hat p_h)
( \omega_x \cdot \nabla \hat u_h) \bigr)\nonumber  \\
& &  + f  \bigl( (\hat u_h - \hat z_f)^2 - |x|  (\hat u_h - \hat z_f) \nabla \hat z_f \cdot \omega_x - |x|  \hat F_f \, \nabla \hat p_h
\cdot \omega_x \bigr); \nonumber \\
\label{eq:Hfh}
H_{f,h}  & = & \frac{1}{f} \bigl( (\nabla \hat p_h \cdot \omega_x) \nabla \hat u_h + (\nabla \hat u_h  \cdot \omega_x) \nabla \hat p_h \bigr)   
 \\ & & -  \frac{2}{f^2}  (\nabla \hat u_h \cdot \omega_x) (\nabla \hat p_h \cdot \omega_x) \,   \nabla_T f.\nonumber 
\end{eqnarray}
Next, let  $\bar h_{f,h}, \, \bar H_{f,h} \in \mathcal{S}^N$ be given by
\begin{eqnarray*}
\int_0^{2 \pi} \bar h_{f,h}(\phi)  \bar v(\phi) \dee \phi & = &  \int_{B_h} h_{f,h}(x) v(\omega_x) \, \dee x, \quad \forall \bar v \in \mathcal{S}^N; \\
\int_0^{2 \pi} \bar H_{f,h}(\phi) \bar v(\phi) \dee \phi &=& \int_{B_h} H_{f,h}(x) \cdot  v(\omega_x) \omega_x^\perp  \, \dee x, \quad \forall \bar v \in \mathcal{S}^N,
\end{eqnarray*}
where, as above $v(\omega)=\bar v(\phi)$ if $\omega=(\cos(\phi),\sin(\phi))^t$ and $(a_1,a_2)^\perp=(-a_2,a_1)$. \\[1.2mm]
\subsubsection{Boundary form of the shape derivative} Let $\DiscreteboundaryFormOfDerivative\colon \partial B_h \rightarrow \mathbb R$ be defined by
\begin{displaymath}
\DiscreteboundaryFormOfDerivative = \frac{1}{2} (\hat u_h - \hat z_f)^2 f  + \frac{1}{f} \bigl( 1 + \frac{| \nabla_T f|^2}{f^2} \bigr)
(\nabla \hat u_h \cdot \omega_h) (\nabla \hat p_h \cdot \omega_h),
\end{displaymath}
where $\omega_h$ is the outer unit normal to $\partial B_h$. Let $\SecondDiscreteBoundaryFormOfDerivative \in \mathcal S^N$ be given by
\begin{displaymath}
\int_0^{2 \pi} \SecondDiscreteBoundaryFormOfDerivative (\phi) \bar v(\phi) \dee \phi = \int_{\partial B_h} \DiscreteboundaryFormOfDerivative(\omega_x) v(\omega_x) \dee o_x, \quad \forall \bar v \in \mathcal{S}^N.
\end{displaymath}
The functions $\bar h_{f,h}$ and $\bar H_{f,h}$ are approximations to those that appear in the formula \eqref{sdint}. We
therefore define $I_h(\bar f)\colon \mathcal S^N \rightarrow \mathbb R$ by
\begin{equation} \label{approxshape}
\langle I_h(\bar f), \bar v \rangle:= \int_0^{2 \pi} \bigl( \bar h_{f,h}(\phi) \bar v(\phi) + \bar H_{f,h}(\phi) \bar v'(\phi) \bigr) \dee \phi, \quad \bar v \in \mathcal S^N
\end{equation}
as an approximation to $\langle J'(f),v \rangle$.
We further note that one can make the analogous definition based on the boundary form, considering the linear operator
\begin{equation}\label{approxBoundaryShape}
    \int_0^{2\pi} \SecondDiscreteBoundaryFormOfDerivative(\phi) \bar{v}(\phi) \dee \phi  \quad \bar{v} \in\mathcal{S}^N.
\end{equation}

\noindent
Based on \eqref{approxshape} or \eqref{approxBoundaryShape} the construction of a nearly optimal descent direction  $\bar g \in \mathcal S^N$ is given by one of the methods described below: a) a discrete
version of the approach of Section \ref{sec:SteepestDescentNIs2} (see \ref{LF}), b) an application of the Sinkhorn Algorithm from optimal transport
(see  \ref{sec:LipOTImplementation}), or c) a Hilbertian method (see  \ref{sec:H1Implementation}).

\begin{remark}
An inspection of our discrete approach yields that we also could choose the function $\bar f \in W^{1,\infty}_{\mbox{per}} (0,2\pi)$ instead of $\bar{f} \in \mathcal{S}^N$, and to consider $I_h(\bar f)$ in \eqref{approxshape} as a linear functional on $W^{1,\infty}_{\mbox{per}} (0,2\pi)$, which would correspond to variational discretisation \cite{Hin05} of our shape optimisation problem. However, the evaluation of integrals through the appearance of the functions $f$ and $\nabla_Tf$ in general requires quadrature rules. In the variational discretisation approach this could be accomplished with replacing $\bar f$ by its Lagrange interpolation, thus leading to the approach proposed in the present section. 
\end{remark}

\subsection{Lipschitz formula} \label{LF}

Since $\bar H_{f,h} \in W^{1,\infty}_{per}(0,2 \pi)$
we may use a discrete version of the approach described in Section \ref{sec:SteepestDescentNIs2} in order to produce an 
approximate direction of steepest descent $\bar g \in \mathcal{S}^N$ as follows:
Fix $\epsilon>0$ and define $G \in \mathcal{S}^N$ by 
\begin{displaymath}
G(\phi_i) := \bar H_{f,h}(\phi_i) - \bar H_{f,h}(0) - \int_0^{\phi_i} \bigl( \bar h_{f,h} - c \bar f \bigr) \dee t, \quad \mbox{ where } c = \frac{\int_0^{2 \pi} \bar h_{f,h} \dee t }{
\int_0^{2 \pi} \bar f \dee t }.
\end{displaymath}
For $i=1,\ldots,N$ and $G_i=G(\phi_i)$ we let 
\begin{eqnarray*}
M_i &:= & \sum_{ \{ j\in\{1,...,N\} | G_j \leq G_i \} } \frac{2\pi}{N} , \quad \beta :=\max \{G_i : M_i < \pi : i =1,...,N\}; \\
O_{\pm} &:= & \{ i \in \lbrace 1,\ldots,N \rbrace : G_i \gtrless \beta \pm \epsilon \},  \,
    O_0 := \lbrace 1,\ldots,N \rbrace \setminus (O_+ \cup O_-), \\
k &:= & \left\{ 
\begin{array}{cl}
0, & O_0 = \emptyset,\\
\frac{|O_+|-|O_-|}{|O_0|}, & \mbox{otherwise}.
\end{array}
\right.
\end{eqnarray*}
Finally, let $\bar g \in \mathcal{S}^N$ be defined by 
\begin{equation}\label{eq:discreteFormula}
\bar g(\phi_i) = \alpha + \frac{1}{2} \frac{2\pi}{N}\sum_{j=1}^i \chi_{j \in O_-} + \chi_{j-1 \in O_-} - \chi_{j \in O_+}- \chi_{j-1 \in O_+} + k \left(\chi_{ j \in O_0}+ \chi_{ j-1 \in O_0}\right),
\end{equation}
where $\alpha$ is chosen so that $\int_0^{2\pi} \bar{f} \bar{g} = 0$ and we identify $0 \cong N$.
We now make some remarks on this discretisation:
\begin{itemize}
    \item  The sets $O_+,\, O_-$ and $O_0$ are not necessarily the natural discrete version of their counterparts in \eqref{defO}, 
    this is chosen to avoid the need to find the points which are identically equal to $\beta$ and allow us to give the function $\bar g$ as a discrete function in $\mathcal{S}^N$.
    \item
    It may be preferable to choose the $\epsilon>0$ to depend on the discretisation and current state.
    For our experiments we take
    \[
        \epsilon = \frac{3}{2N}\left( \max_{i=1,...,N} G_i - \min_{i=1,...,N}G_i \right)
    \]
\end{itemize}

\subsection{Sinkhorn algorithm}\label{sec:LipOTImplementation}
In this section we aim to use a discrete version of the observation in  Section \ref{sec:OT} to obtain an approximate direction of steepest descent $\bar g \in \mathcal{S}^N$.
Abbreviating $q_{f,h}:= \bar h_{f,h} - \bar H'_{f,h} - c \bar f$ we may write
\begin{displaymath}
\langle I_h(\bar f), \bar v \rangle = \int_0^{2 \pi} q_{f,h} \bar v \dee \phi.
\end{displaymath}
Denoting by $\lbrace \varphi_1,\ldots,\varphi_N \rbrace$ the standard
nodal basis of $\mathcal S^N$, where due to periodicity we require only $N$ basis functions, we have $\bar v = \sum_{i=1}^N \bar v(\phi_i) \varphi_i$.
If $\int_0^{2\pi} \bar v \bar f \dee \phi =0$,
\begin{displaymath}
\langle I_h(\bar f), \bar v \rangle = \sum_{i=1}^N a_i \bar v(\phi_i), \quad \mbox{ where } \quad a_i:= \int_0^{2 \pi} q_{f,h} \varphi_i \dee \phi, \quad i=1,\ldots,N.
\end{displaymath}
Notice that $a$ satisfies $\sum_{i=1}^N a_i=0$ because $\int_0^{2\pi} q_{f,h} \dee \phi =0$.
Let $N^\pm:= \lbrace i \in \lbrace 1,\ldots,N \rbrace \, | \, a_i \gtrless 0 \rbrace$.
Let us abbreviate $C_{ij}=d(\phi_i,\phi_j)$, where
\begin{displaymath}
d(\phi,\tilde \phi)=\arccos(\cos(\phi-\tilde \phi)), \quad \phi, \tilde \phi \in [0, 2\pi]
\end{displaymath}
and set $R(C):= \{ \xi \in \R^N\, |\, \xi_i - \xi_j \leq C_{ij}, (i,j) \in N^+\times N^-\}$ as well as
\[
    U(a) := \left\{ P \in \R^{N \times N} | P \geq 0,\, \sum_{j \in N^-} P_{ij} = a_i,\, i \in N^+, \, \sum_{i \in N^+}P_{ij} = -a_j,\, j \in N^-\right\}.
\]
We obtain a discrete version of
\eqref{maxmin} via 
\begin{eqnarray}
\max_{\bar v \in \mathcal S^N \cap V_\infty(\bar f)} \langle I_h(\bar f), \bar v \rangle 
& = & \max \left\{ \sum_{i\in N^+} a_i \xi_i + \sum_{j \in N^-} a_j \xi_j \,  | \,\xi \in R(C) \right\} \nonumber \\
& = & \min \left\{ \sum_{i \in N^+}\sum_{j \in N^-} C_{ij}P_{ij} \, | \, P \in U(a) \right\},\label{eq:DiscreteDualityRelation}
\end{eqnarray}
where the latter is a discrete optimal transport problem and the former its dual, which we solve with the help of the Sinkhorn algorithm.

For $\delta>0$, the Sinkhorn algorithm minimises the regularised quantity
\[
    \sum_{i \in N^+}\sum_{j \in N^-} C_{ij}P_{ij} + \delta P_{ij}(\log(P_{ij})-1).
\]
For notational convenience, let $u^l,\, v^l \in \R^N$ for $l \geq 0$.
Letting $K_{ij} = \exp(-\frac{1}{\delta} C_{ij}),\, u^0_i = 1$ and $v^0_j=1$ for $i \in N^+,\, j \in N^-$ and $0$ otherwise, the Sinkhorn iteration is given by
\[
    u_i^{l+1} = \frac{a_i}{(Kv^l)_i}, i \in N^+, \quad v_j^{l+1} = \frac{-a_j}{(K^t u^{l+1})_j}, j \in N^-
\]
for $l \geq 0$. 
The matrix $(u_i^l K_{ij}v_j^l)_{(i,j)\in N^+\times N^-}$ corresponds to the primal solution, that is, will approximate the optimal $P_{ij}$.
The vectors $(\delta \log(u^l_i))_{i \in N^+}$ and $(\delta \log(v^l_j))_{j \in N^-}$ correspond to the dual solution in this algorithm, that is, they will approximate the optimal $\xi$, with $\xi_i \approx \delta \log(u^l_i),\, \xi_j \approx -\delta \log(v^l_j)$, $(i,j) \in N^+\times N^-$.
For further details on the Sinkhorn algorithm, we refer the reader to \cite[Section 4.2]{PeyCut19}.

We set $\delta = 0.05$ and stop the iterations when either $l = 2000$ or $\frac{1}{|N^+|} \sum_{i \in N^+} |a_i - \sum_{j \in N^-} u^l_i K_{ij} v^l_j  | \leq 10^{-6}$ 
and $\frac{1}{|N^-|} \sum_{j \in N^-} |-a_j - \sum_{i\in N^+} u^l_i K_{ij} v^l_j  | \leq 10^{-6}$. \\
We finally define $\bar g \in \mathcal S^N$ by assigning the following values at the vertices $\phi_i$:
\begin{eqnarray*}
    \bar g(\phi_i) & = & \alpha + \inf_{j \in N^-} \left( -\delta \log(v^l_j) + C_{ij}\right), \quad i  \in N^+, \\
    \bar g(\phi_j) & = & \sup_{i \in N^+} \left(\bar g(\phi_i) - C_{ij}\right), \quad j \in N^-, \\
    \bar g(\phi_k) & = & \alpha+ \inf_{j \in N^-} \left( -\delta \log(v_j^l) + d{(\phi_k,\phi_j)} \right), \quad k \in \lbrace 1,\ldots,N \rbrace \setminus (N^+ \cup N^-),
\end{eqnarray*}
where $\alpha$ is chosen to ensure that $\int_0^{2 \pi} \bar f \bar g \dee \phi = 0$.
We note that this method to assign $\bar g$ at the vertices is not the natural way.
We chose the method given above as we found it to give slightly preferable results for $\bar g$ both in terms of the shape and of the evaluation of $\int_0^{2 \pi} q_{f,h} \bar g \dee \phi$.
As we see, particularly from \eqref{eq:DiscreteDualityRelation}, a more natural way would be to directly assign the dual variables $\alpha + \delta \log(u^l_i)$ to $\bar g(\phi_i)$ for $i \in N^+$ and $\alpha -\delta \log(v^l_i)$ to $\bar g(\phi_j)$ for $j \in N^-$, where $\alpha$ is again an additive constant to ensure the integral condition.

\subsection{$H^1$ minimising direction}\label{sec:H1Implementation}
As outlined in Remark \ref{rem:IL05} 
our Lipschitz direction $\bar g$ from Proposition \ref{prop:minimiser} can be obtained as the uniform $p$-limit of the minimisers  $v^*$ of
\[
    v \in \mathcal{S}^N \mapsto \int_0^{2\pi}  \frac{1}{p} | v'|^p + \bar h_{f,h} v + \bar H_{f,h} v' \dee \phi
\]
such that $\int_0^{2\pi} v \bar{f} = 0$.
We rescale $\bar g_{p} := \frac{v^*}{\|(v^*)'\|_{L^p(0,2\pi)}}$  so that $\|\bar g_p '\|_{L^p(0,2\pi)} = 1$, i.e. $\bar g_p$ is a direction in the topology induced by the $W^{1,p}$-seminorm.
In this setting the approaches commonly used in the literature so far correspond to the Hilbert space case $p=2$, see \cite[Section 5.2]{ADJ21} for a detailed discussion, and also \cite{EH12,HSU21,HPS15,SISG13,SSW15,SSW16,SW17}, which we here consider as a reference case for comparison of our approach.
Of course, one may also do this for the boundary form of the shape derivative, considering the minimiser of
\[
    v \in \mathcal{S}^N \mapsto \int_0^{2\pi}  \frac{1}{p} | v'|^p + \SecondDiscreteBoundaryFormOfDerivative v \dee \phi
\]
such that $\int_0^{2\pi} v \bar{f} = 0$.

\section{Numerical Experiments}

The numerical experiments carried out in this section combine the following Armijo--type descent method with one of the choices for a descent direction
described in the previous section:

\begin{algorithm}
Given $\bar f \in \mathcal S^N$\;
Solve for $\hat{u}_h$\;
Set $E = \frac{1}{2}\int_{B_h} (\hat{u}_h-\hat{z}_f)^2 f^2$\;
\For{$j=1,...,{\rm maxIt}$}{
    Solve for $\hat{p}_h$\;
    Construct descent $\bar g$\;
    Set $fOld = \bar f$\;
    \For{ $\sigma \in \{ 1/16,1/32,1/64,...\}$, and $\sigma \ge 10^{-8}$}{
        Set $\bar f = fOld + \sigma \bar g$\;
        Solve for $\hat{u}_h$\;
        \If{$\frac{1}{2}\int_{B_h} (\hat{u}_h - \hat{z}_f)^2 f^2 < E + 10^{-5} \sigma \langle I_h(\bar f) , \bar g\rangle $}{
            Set $ E = \frac{1}{2}\int_{B_h} (\hat{u}_h - \hat{z}_f)^2 f^2 $\;
            \Break\;
        }
    }
}
\caption{Our implemented Armijo algorithm}\label{alg:Armijo}
\end{algorithm}

We set ${\rm maxIt} = 250$.
In our numerical implementation, whenever we set $f$, we rescale it to have the same square integral as the original domain.
The images of the grids are created with ParaView \cite{paraviewReference} and our finite element methods for state and adjoint equations are performed with DUNE \cite{duneReference}.
The boundary has discretisation with $N = 512$ and the triangulation $B_h$ is shown in Figure \ref{fig:Triangulation}.
\begin{figure}
    \centering
    \includegraphics[width=.6\linewidth]{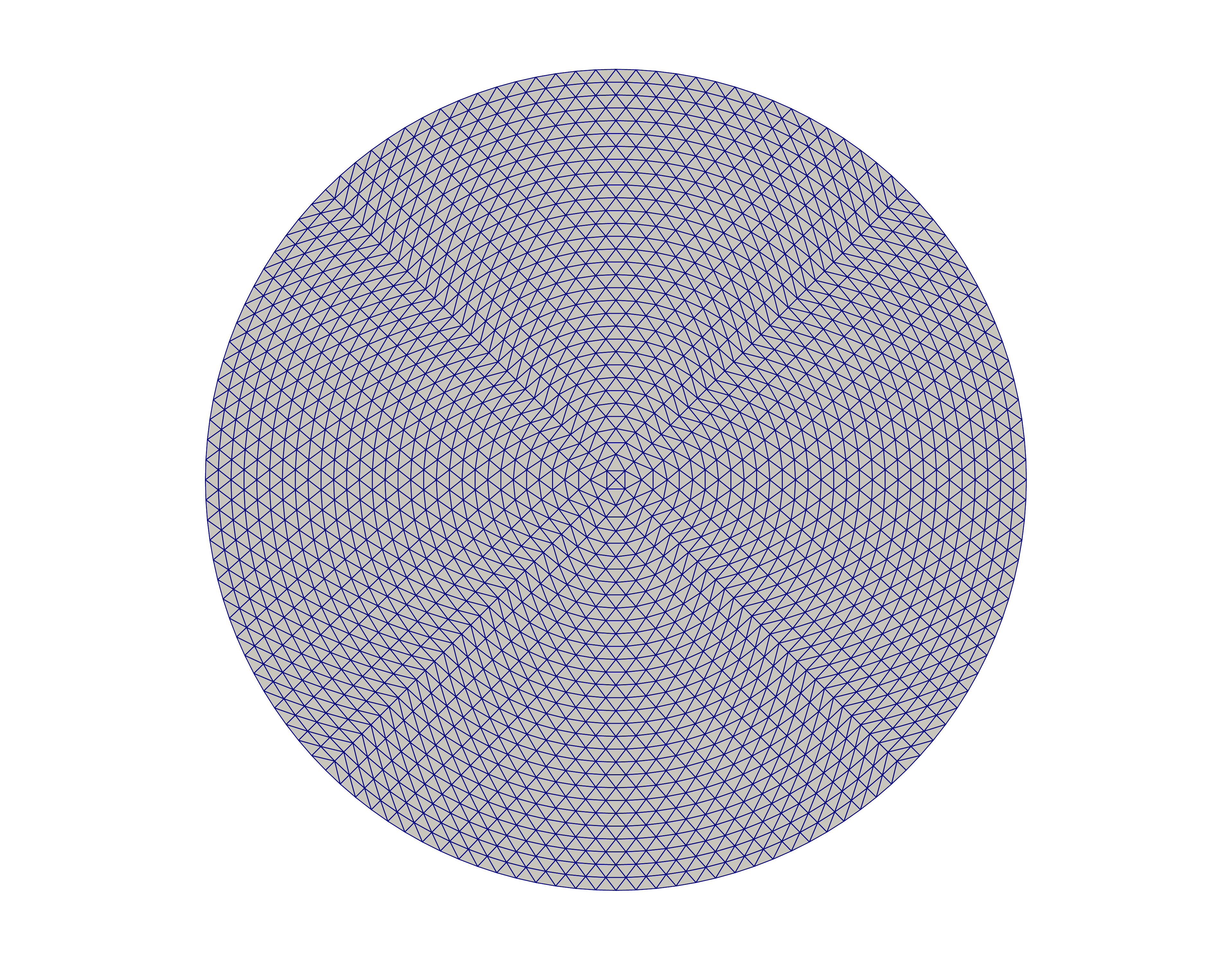}
    \caption{Triangulation of the computational domain.}
    \label{fig:Triangulation}
\end{figure}

We use a log scale to plot the graphs for the energy throughout the iterations of the experiments.
When the energy is not expected to vanish, as in the experiments in Sections \ref{sec:Experiment3} and \ref{sec:Experiment2}, we take away the lowest energy value attained by any of the experiments from all of the data, this value appears in the $y$ axis label of the graphs.
For the experiments in Sections \ref{sec:Experiment1} and \ref{sec:Experiment4}, the energy is expected to become very close to zero.
In the legends for the graphs, we abbreviate 'optimal transport' to 'OT'.

When we give the images of the domains, we will display the image of $\mathcal{T}_h$ under the map $\Phi_{f}$.
It is important to recall that this is not the computational domain.
Since the mesh does not deform, we do not need to worry about degeneration, however it is worth mentioning that when the image $\Phi_f(\mathcal{T}_h)$ has a low quality, one will expect that the finite element approximations become poor, in much the same way as for a degenerate mesh.
In practice, when this happens one may wish to remesh $\mathcal{T}_h$ in certain ways, or the boundary mesh.

\subsection{An experiment with $F=0$}\label{sec:Experiment3}
For this experiment, we set $F(x_1,x_2) = 0$ and $z(x_1,x_2) = |x_1+x_2| + |x_1-x_2|$.
Since $F = 0$ it follows that $u=0$, therefore when considering the boundary form of the shape derivative which appears in \eqref{eq:BoundaryDerivativeGeneralCase}, we see that when the boundary of the domain is in a level-set of $z^2$, the energy will be critical over directions which preserve the volume.
When starting with $f= 1$, we expect the final domain to be the square $\left(-\frac{\sqrt{\pi}}{2},\frac{\sqrt{\pi}}{2} \right)^2$.
After $250$ iterations, the Lipschitz optimal transport method with the boundary form of derivative gives the domain on the left of Figure \ref{fig:Experiment3FinalDomain}.
The method using the Lipschitz formula with boundary form of derivative terminated after $103$ iterations, the domain at this point is given in the middle of Figure \ref{fig:Experiment3FinalDomain}.
At this point, the resulting shape has very low energy and is very close to the shape we expected to be minimal.
The result of $250$ iterations of the $H^1$ method is given on the right of Figure \ref{fig:Experiment3FinalDomain}.
\begin{figure}
    \centering
    \includegraphics[width=.3\linewidth]{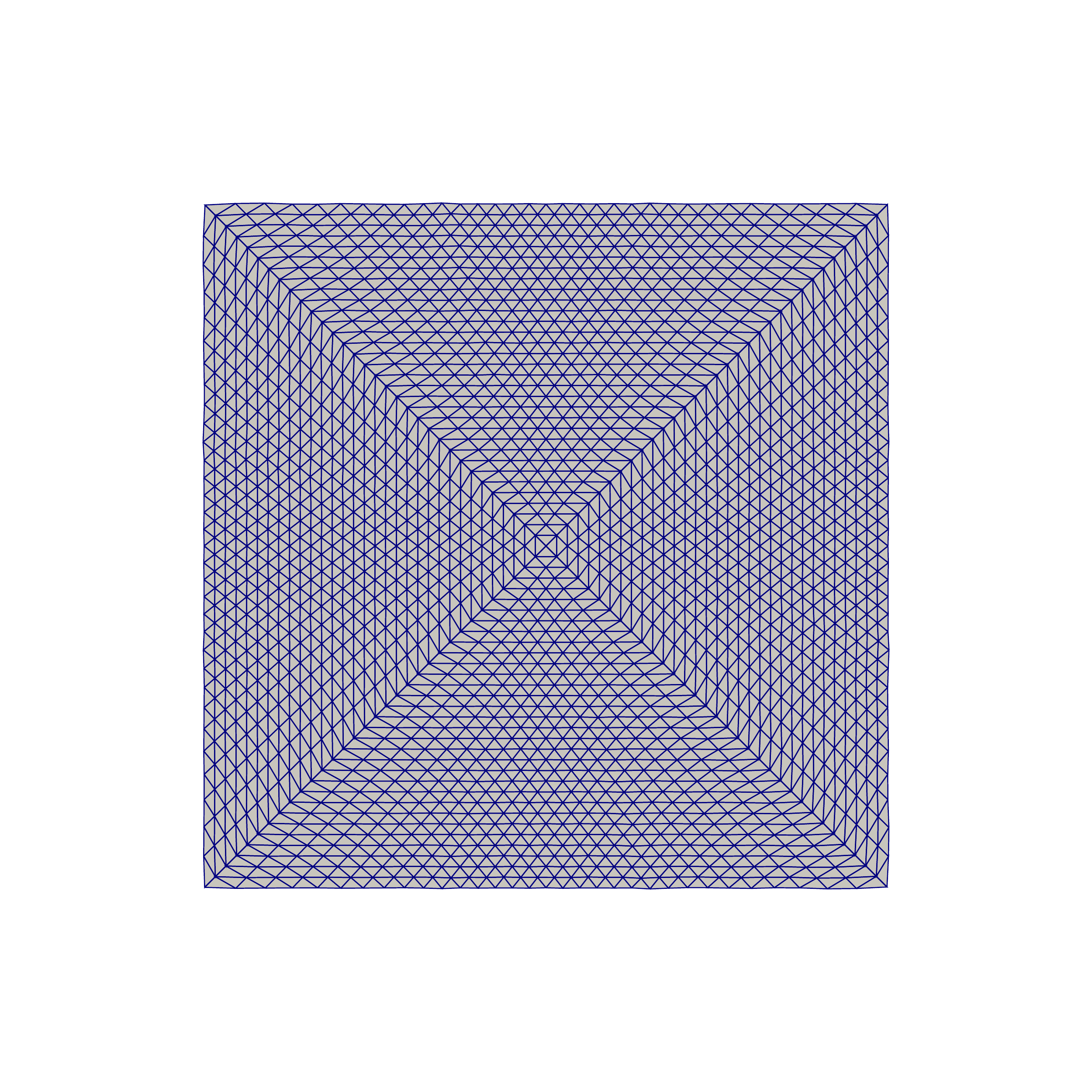} \hspace{0.02\linewidth}
    \includegraphics[width=.3\linewidth]{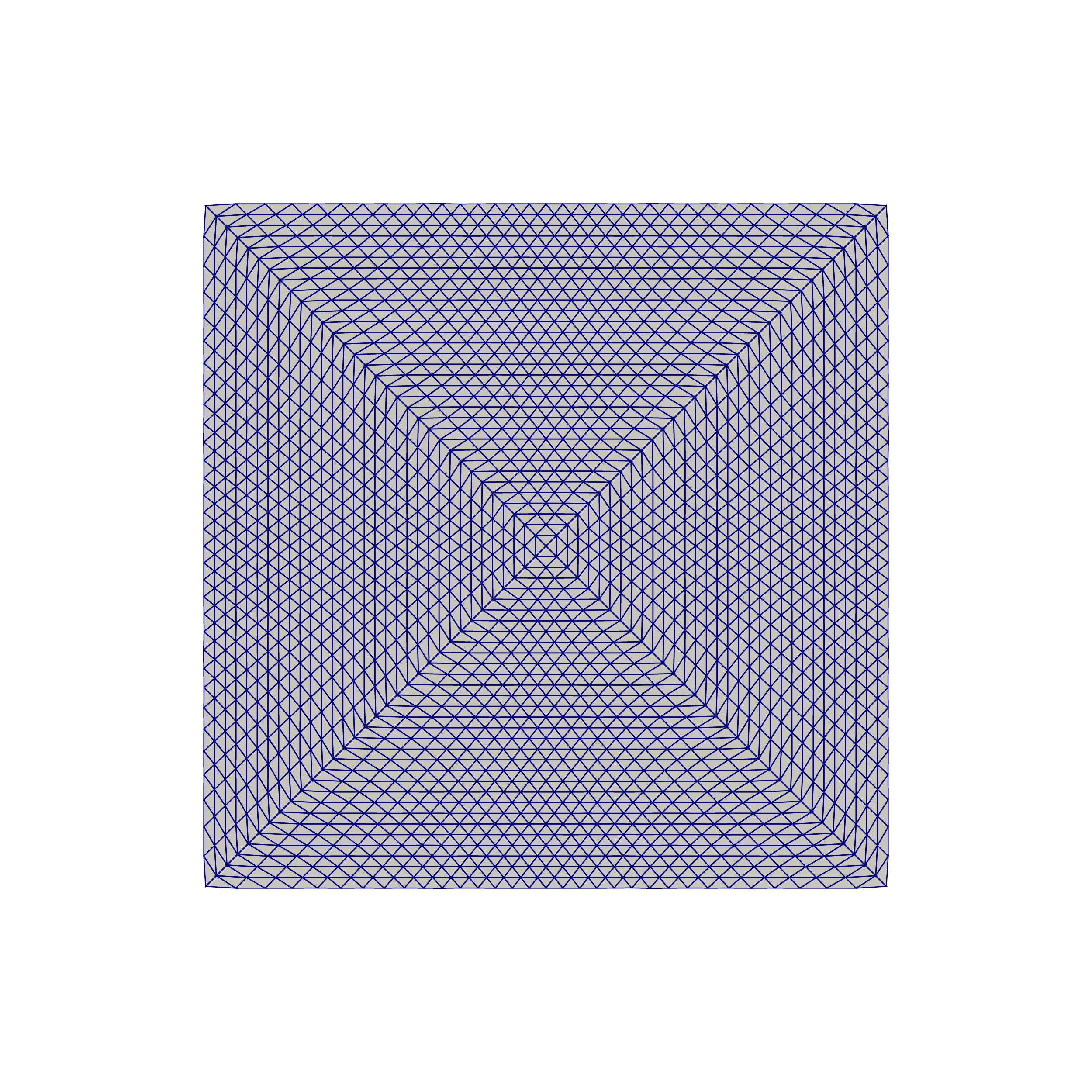} \hspace{0.02\linewidth}
    \includegraphics[width=.3\linewidth]{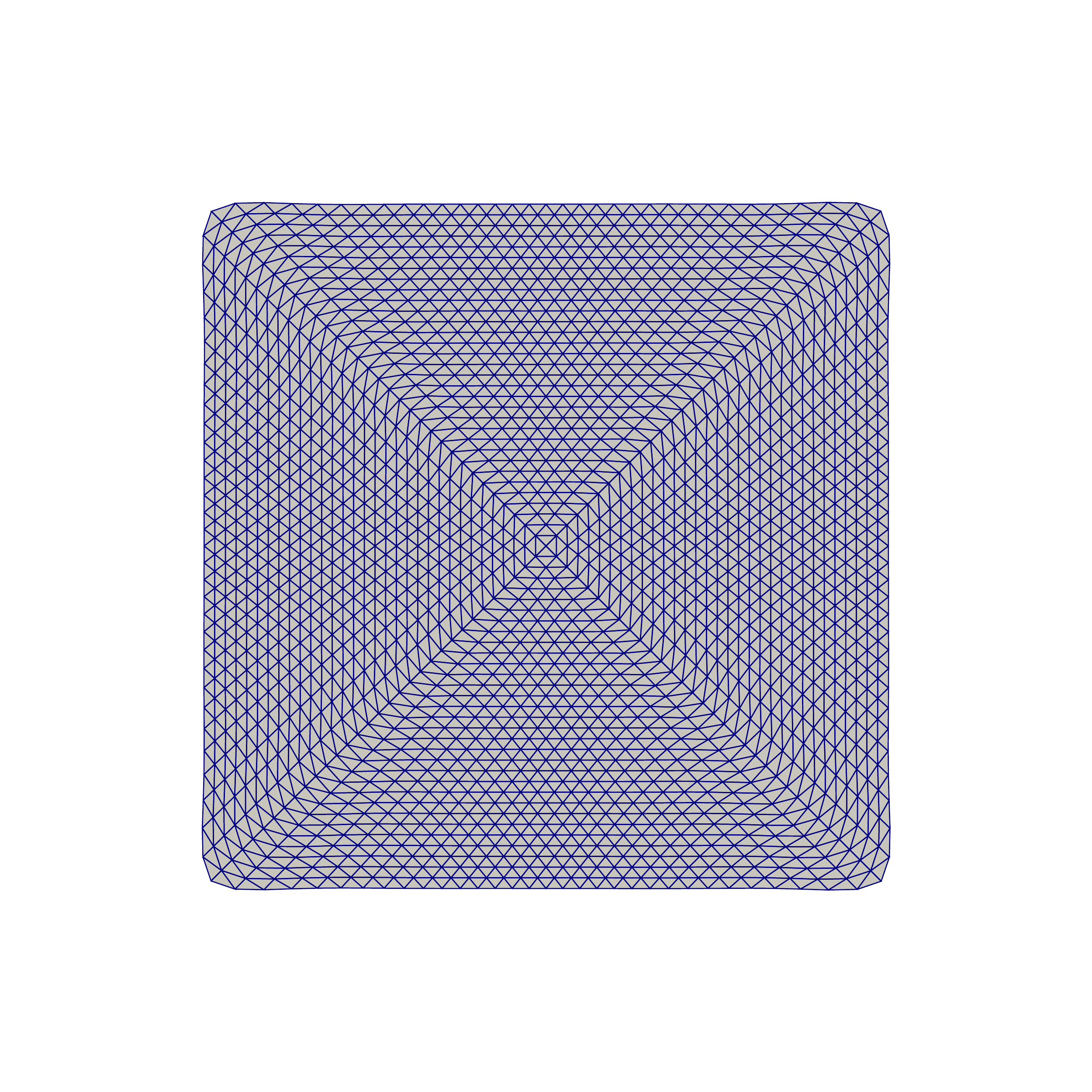}
    \caption{Final domains for the experiment in Section \ref{sec:Experiment3} with Lipschitz optimal transport descent (left), Lipschitz formula descent (middle) and $H^1$ descent (right).}
    \label{fig:Experiment3FinalDomain}
\end{figure}
A graph of energy throughout the iterations is given in Figure \ref{fig:Experiment3Graph}.
\begin{figure}
    \centering
    \includegraphics[width=.9\linewidth]{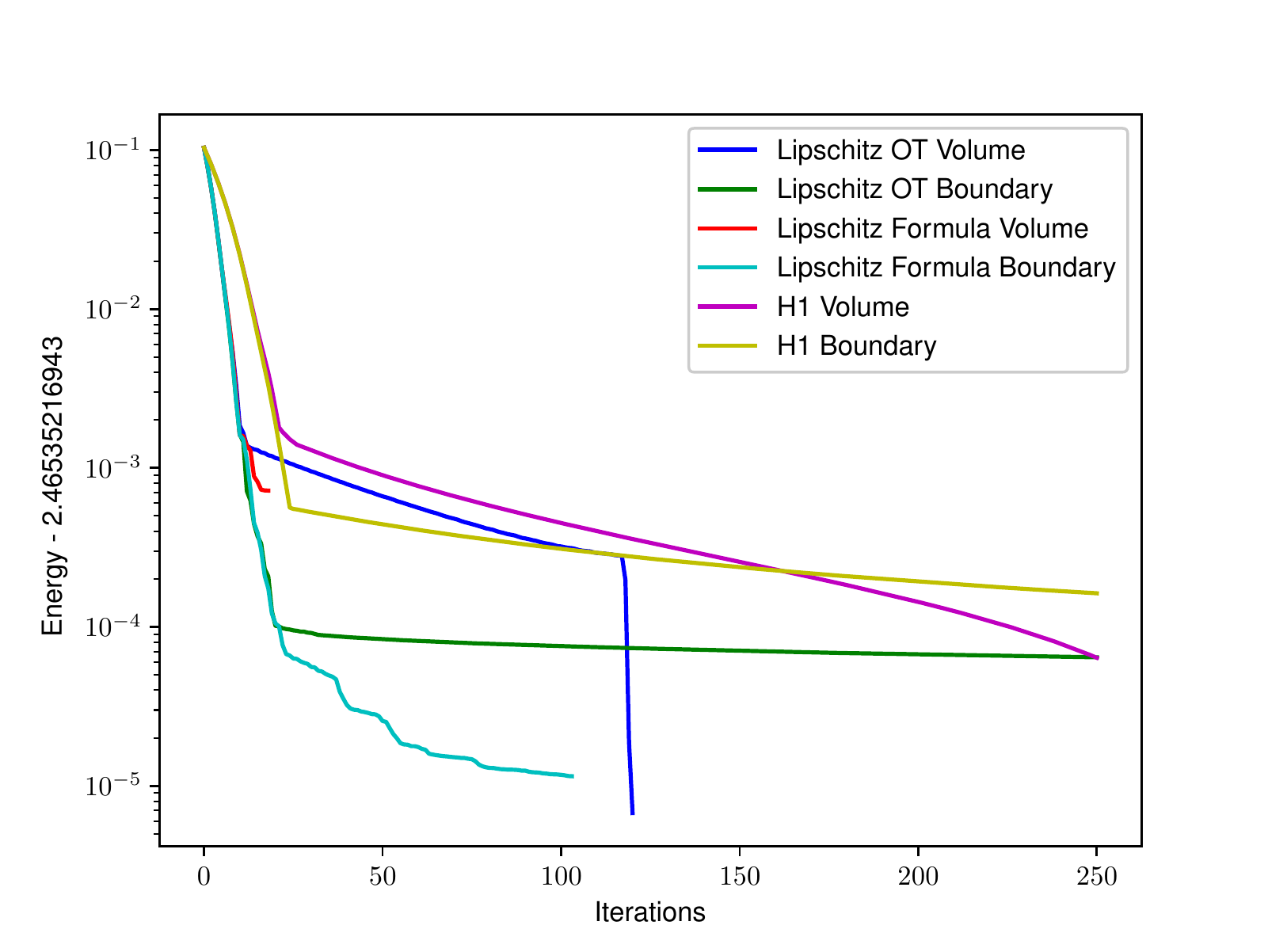}
    \caption{Graph of the energy for the iterates in the experiment in Section \ref{sec:Experiment3}}
    \label{fig:Experiment3Graph}
\end{figure}
A graph of the magnitude of the discrete directional derivative throughout the iterations is given in Figure \ref{fig:Experiment3DiscreteDerivative}.
\begin{figure}
    \centering
    \includegraphics[width=.9\linewidth]{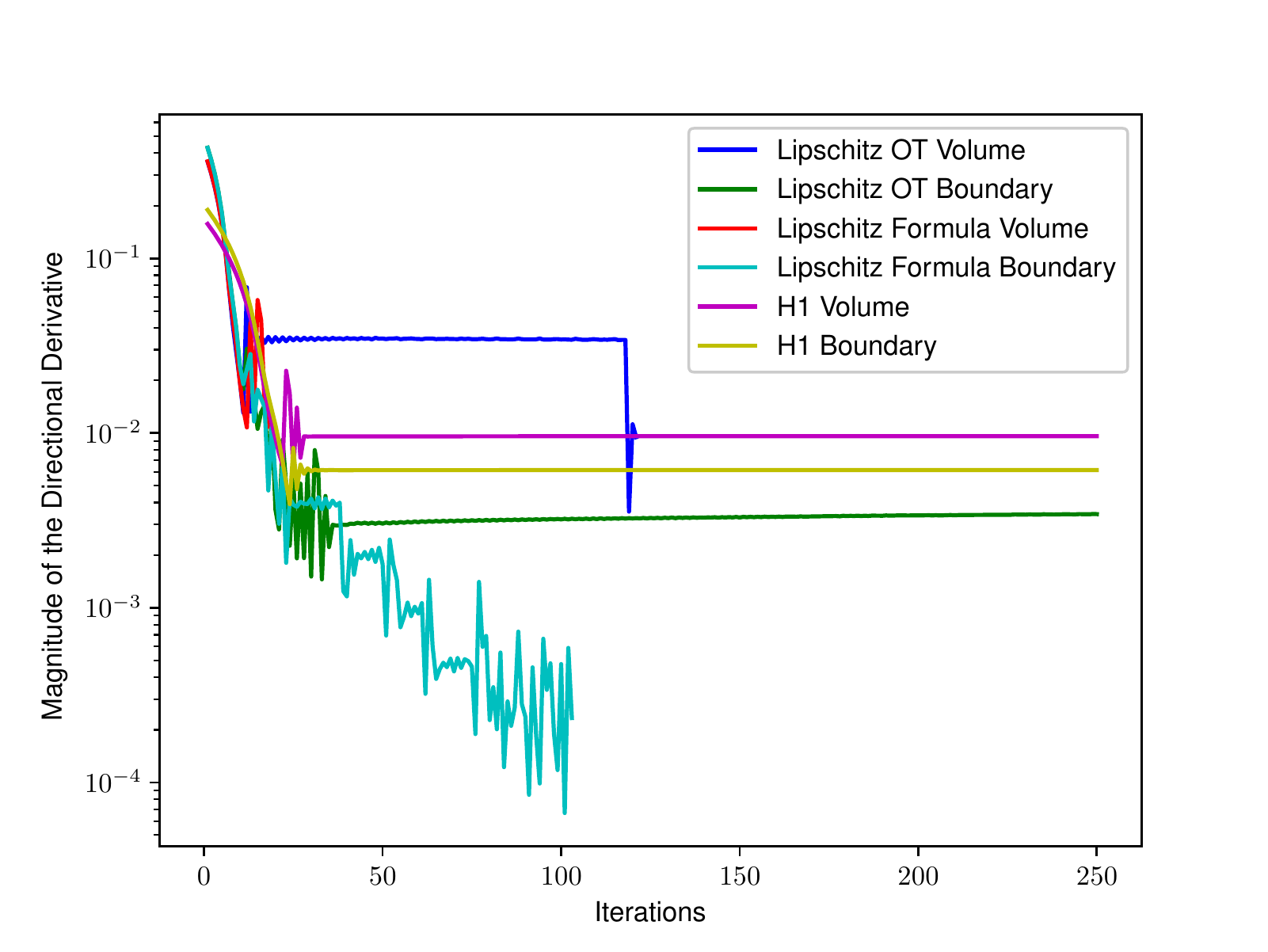}
    \caption{Graph of the magnitude of the discrete directional derivative for the iterates in the experiment in Section \ref{sec:Experiment3}}
    \label{fig:Experiment3DiscreteDerivative}
\end{figure}
The Lipschitz formula method with volume form of the derivative terminated prematurely after only $18$ steps, far from the minimum shape.
The Lipschitz optimal transport method with volume form of derivative terminated after $121$ steps, where we see that this energy has already become very low.
One may see that the corners from both the Lipschitz methods are highly developed, whereas they are rather curved for the $H^1$ method, this is highlighted in Figure \ref{fig:Experiment3FinalCorners} which gives a zoom on the corners of Figure \ref{fig:Experiment3FinalDomain}.
\begin{figure}
\centering
\includegraphics[width=.31\linewidth]{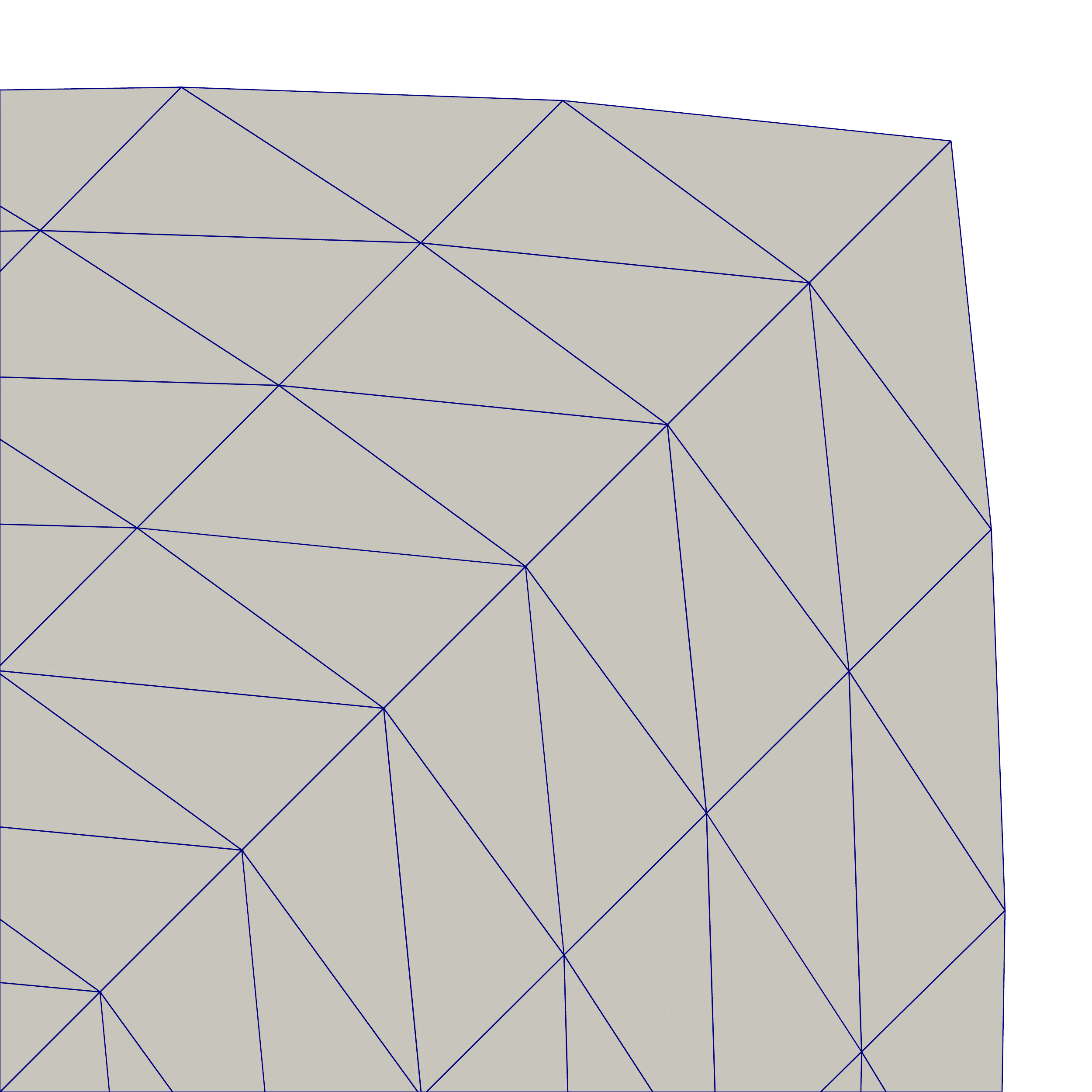}
\hspace{0.01\linewidth}
\includegraphics[width=.31\linewidth]{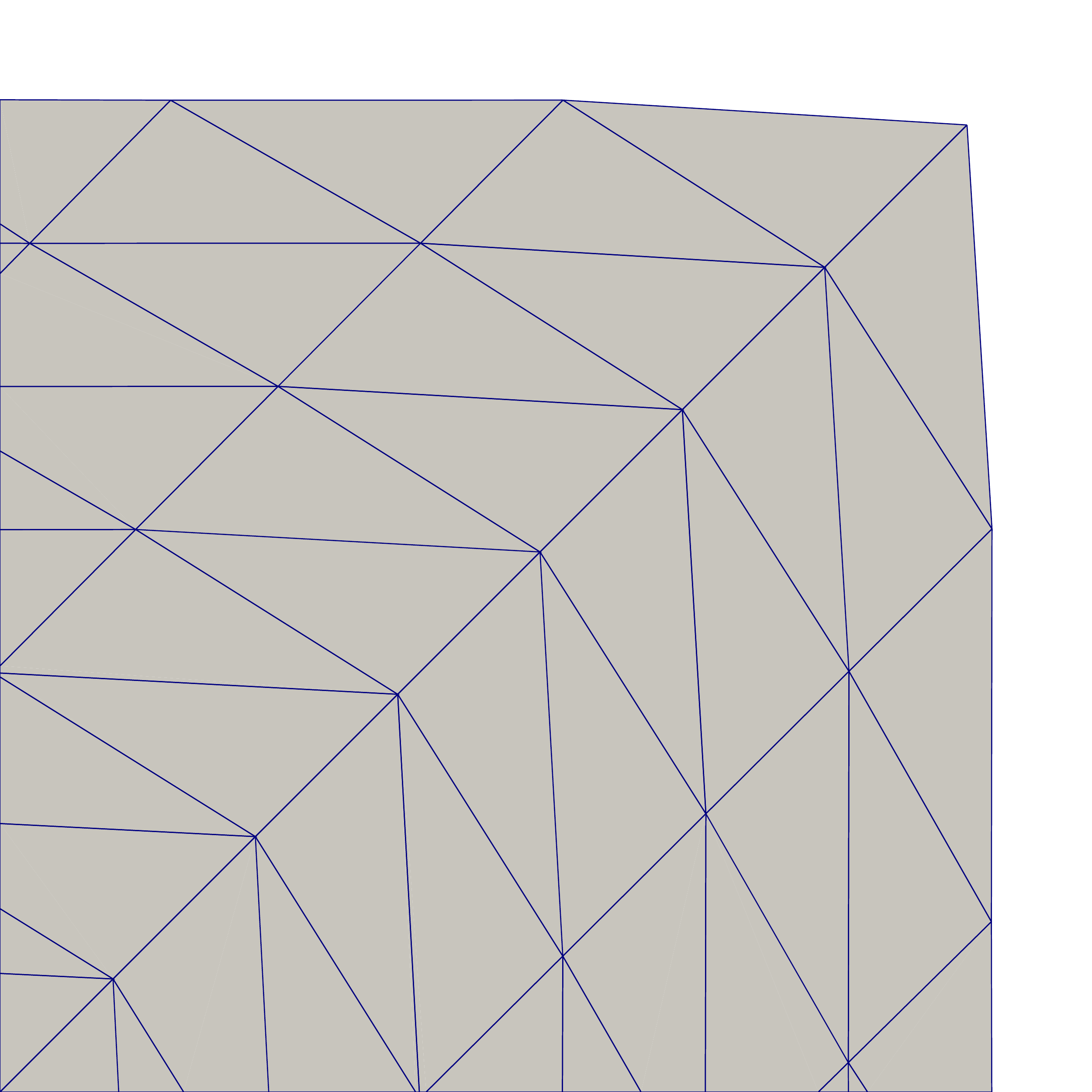}
\hspace{0.01\linewidth}
\includegraphics[width=.31\linewidth]{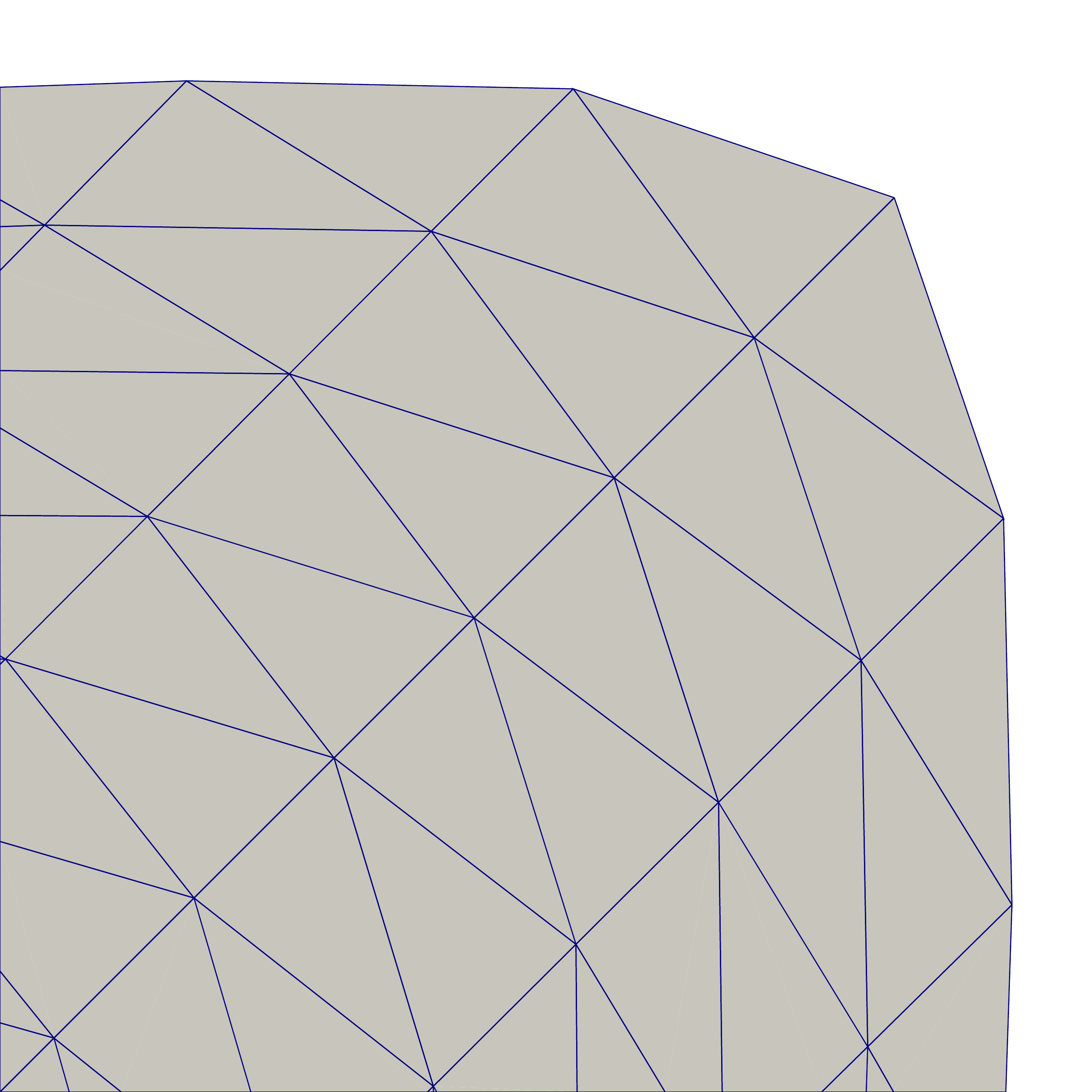}
\caption{Zoom on the top right 'corners' of the final domains for the experiment in Section \ref{sec:Experiment3} with Lipschitz optimal transport descent (left), Lipschitz formula descent (middle) and $H^1$ descent (right).}
\label{fig:Experiment3FinalCorners}
\end{figure}

\subsection{An experiment with $ -\Delta z = 4 F$} \label{sec:Experiment2}
For this experiment we set $F(x_1,x_2) = 1$ and $z(x_1,x_2) = 1-x_1^2-x_2^2$.
We notice that $-\Delta z = 4F$ and that $z$ vanishes on the unit circle.
We start this experiment with $f$ representing the square  $\left(-\frac{\sqrt{\pi}}{2},\frac{\sqrt{\pi}}{2} \right)^2$, which is shown in Figure \ref{fig:Experiment2Domain}.
\begin{figure}
    \centering
    \includegraphics[width=.6\linewidth]{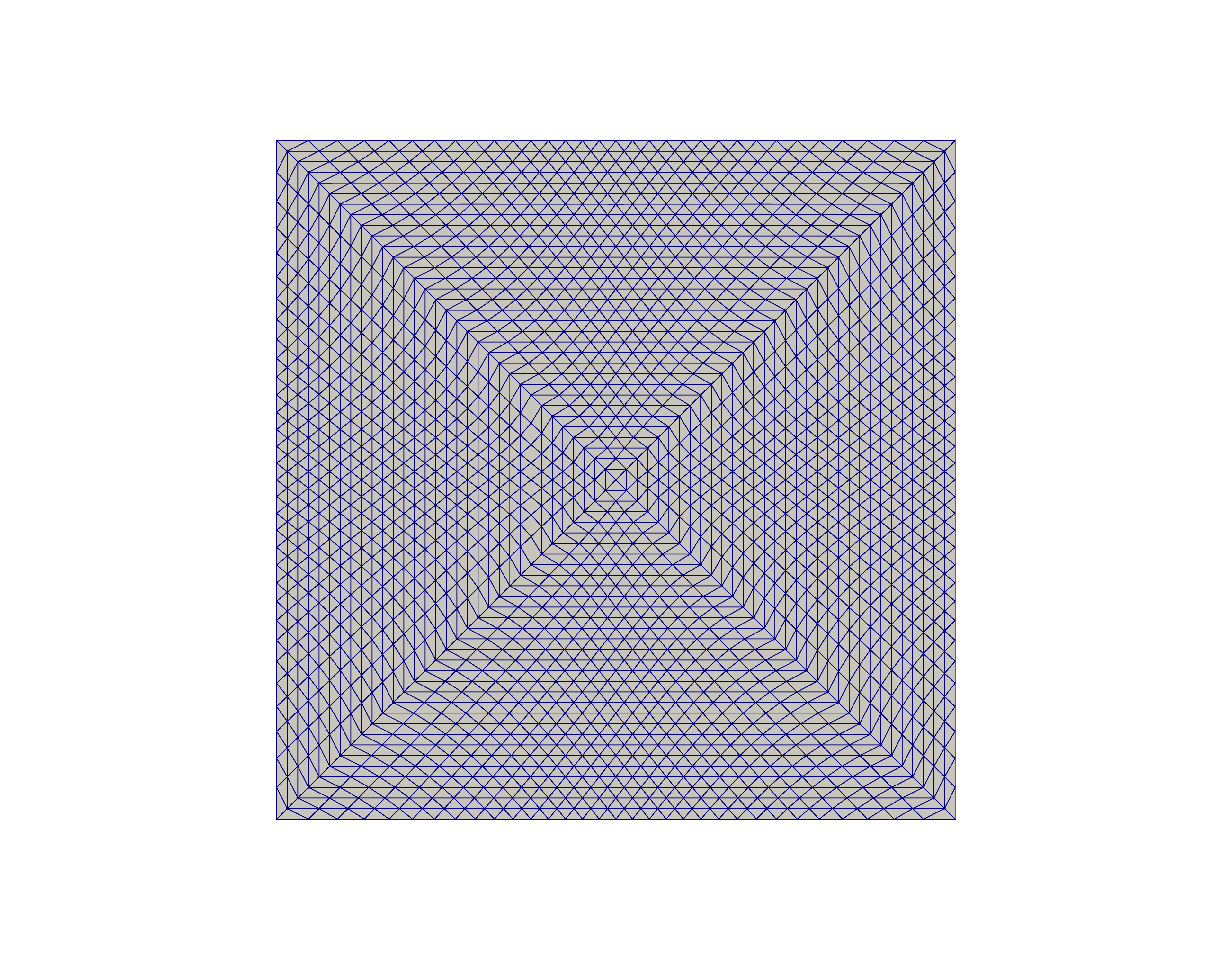}
    \caption{Initial domain for the experiment in Section \ref{sec:Experiment2}.}
    \label{fig:Experiment2Domain}
\end{figure}

In this experiment we provide the domain after $15$ iterations, this appears in Figure \ref{fig:Experiment2IntermediateDomain}.
This is given as such comparisons are of interest in practical applications, where computation time is a limiting factor.
\begin{figure}
\centering
\includegraphics[width=.45\linewidth]{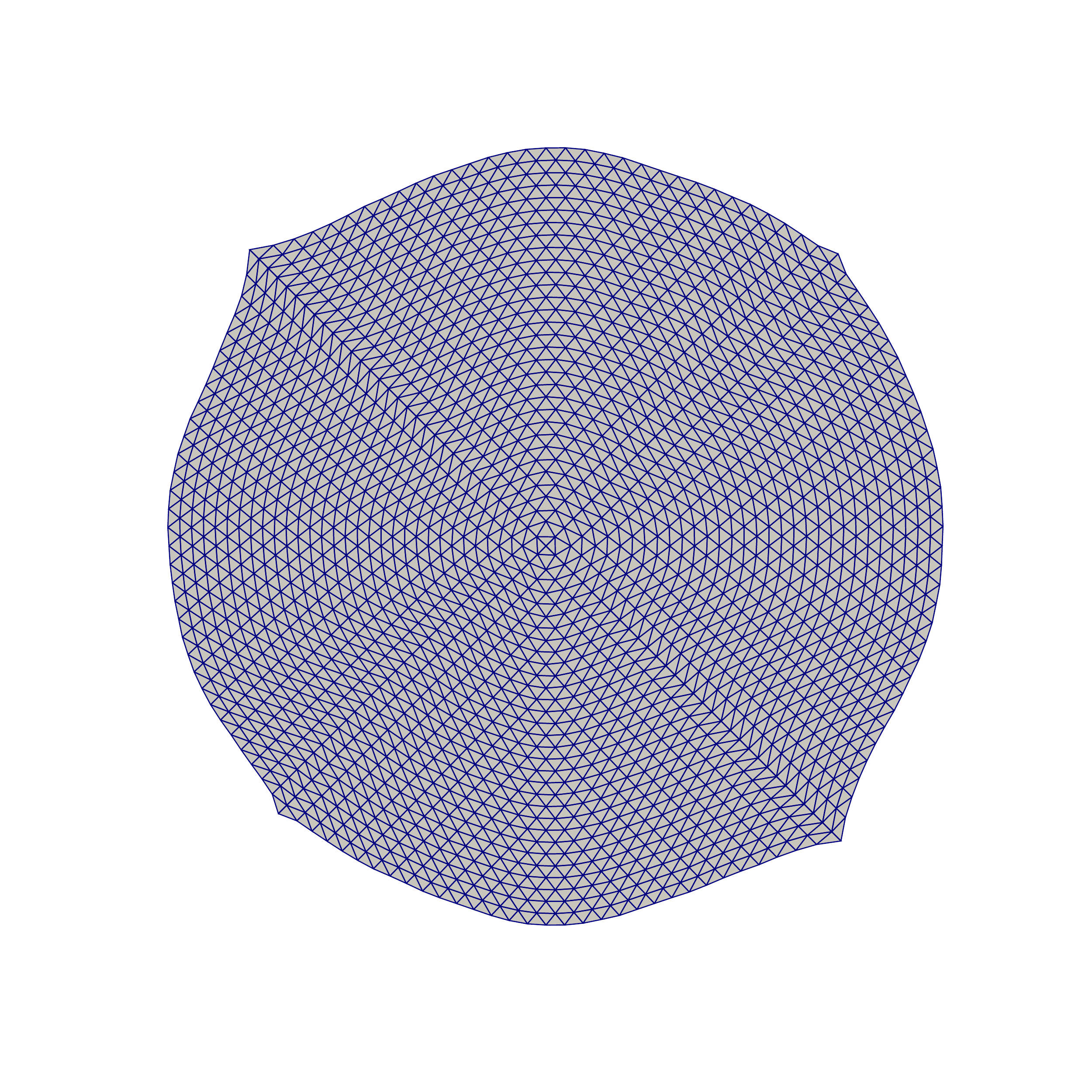}
\hspace{0.05\linewidth}
\includegraphics[width=.45\linewidth]{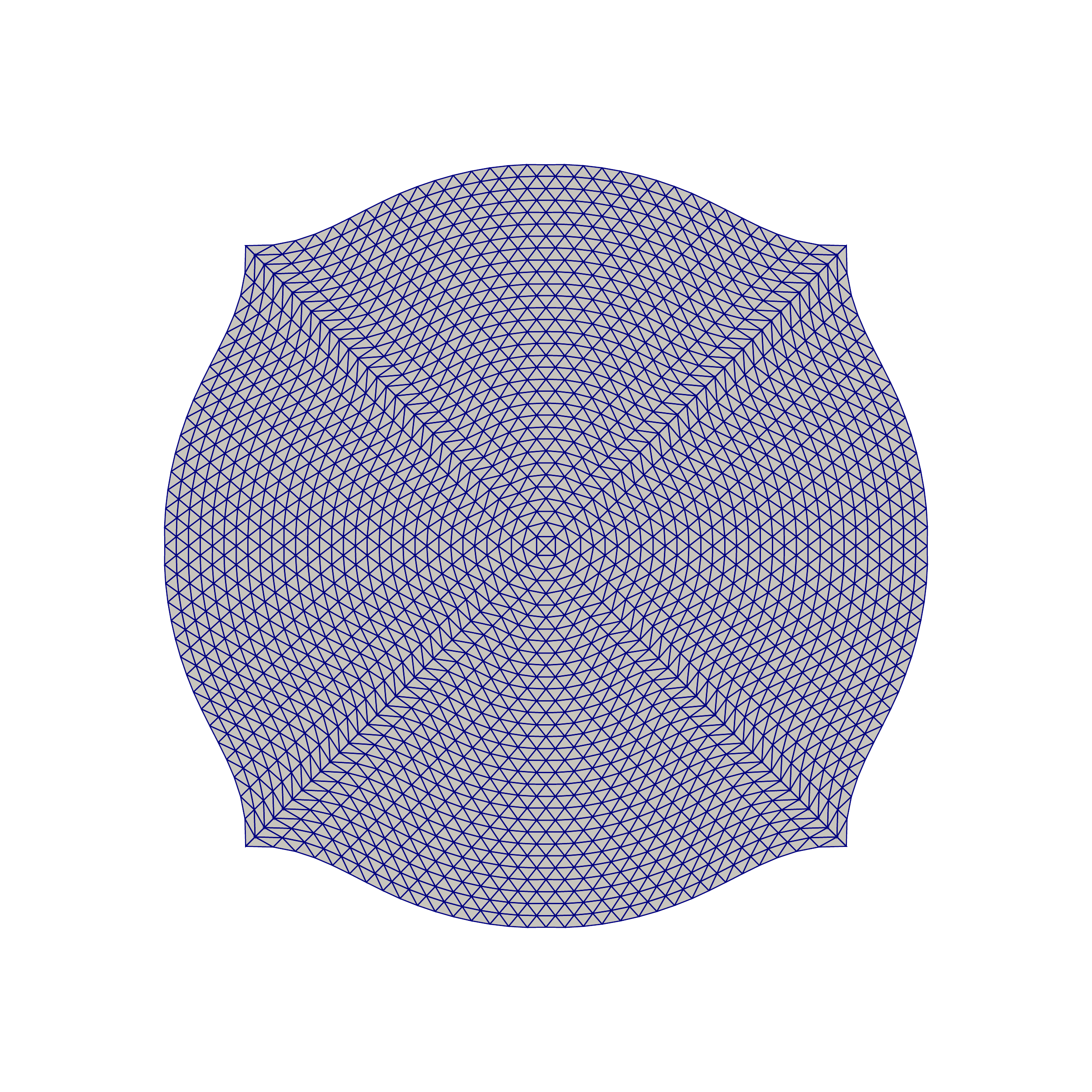}
\vspace{0.075\linewidth}
\includegraphics[width=.45\linewidth]{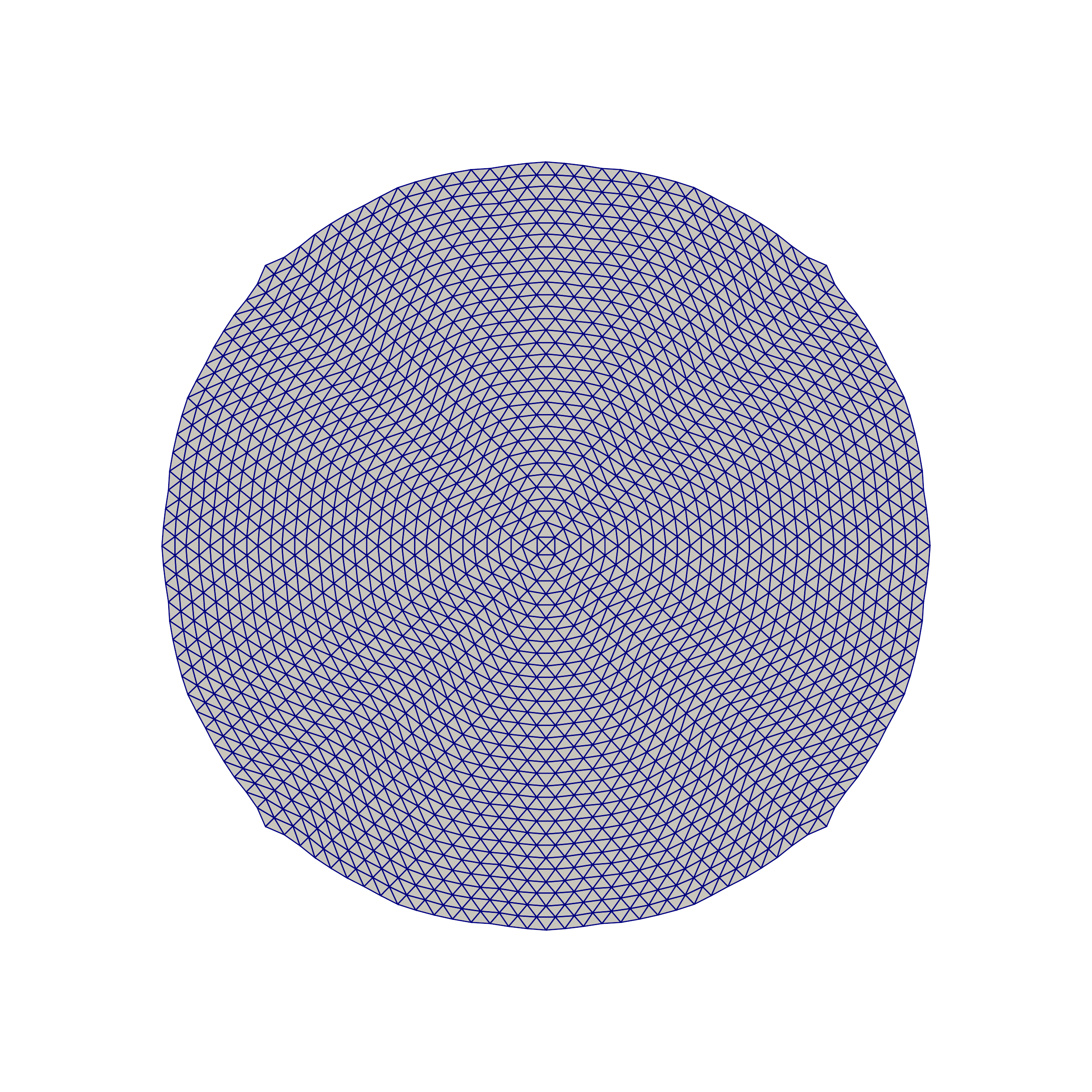}
\hspace{0.05\linewidth}
\includegraphics[width=.45\linewidth]{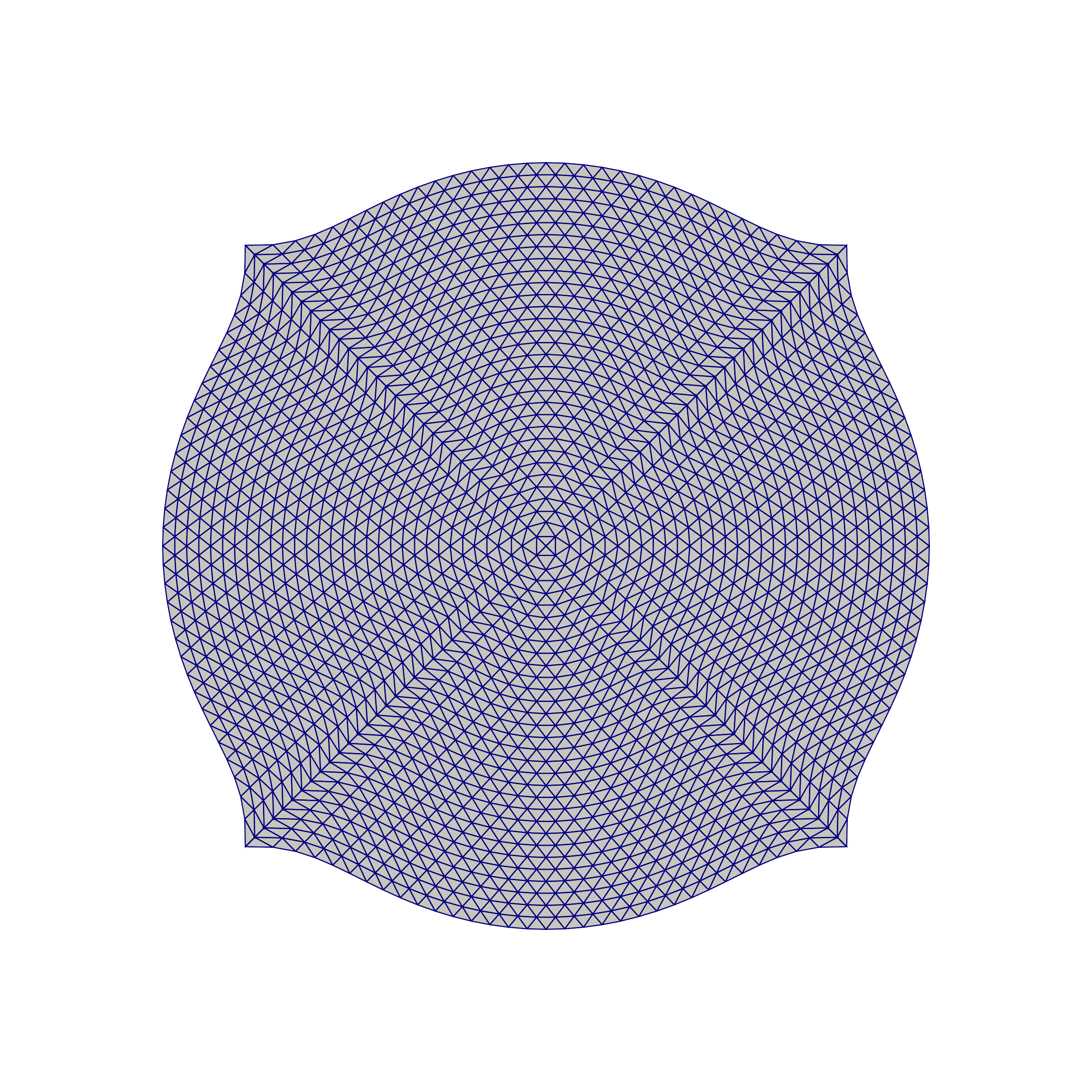}
\caption{Domains after 15 iterations for the experiment in Section \ref{sec:Experiment2} with Lipschitz formula descent (left) and $H^1$ descent (right) with the volume form (top) of the shape derivative and the boundary form (bottom).}
\label{fig:Experiment2IntermediateDomain}
\end{figure}
We see that even after only $15$ iterations the shapes are close to a circle, with both the Lipschitz methods outperforming the $H^1$ methods significantly.

After $250$ iterations, the Lipschitz optimal transport method with volume form of the shape derivative gives the domain on the left of Figure \ref{fig:Experiment2FinalDomain}.
The $H^1$ method with volume form of the shape derivative prematurely terminated after $31$ iterations and the domain at this point is shown on the right of Figure \ref{fig:Experiment2FinalDomain}.
\begin{figure}
\centering
\includegraphics[width=.45\linewidth]{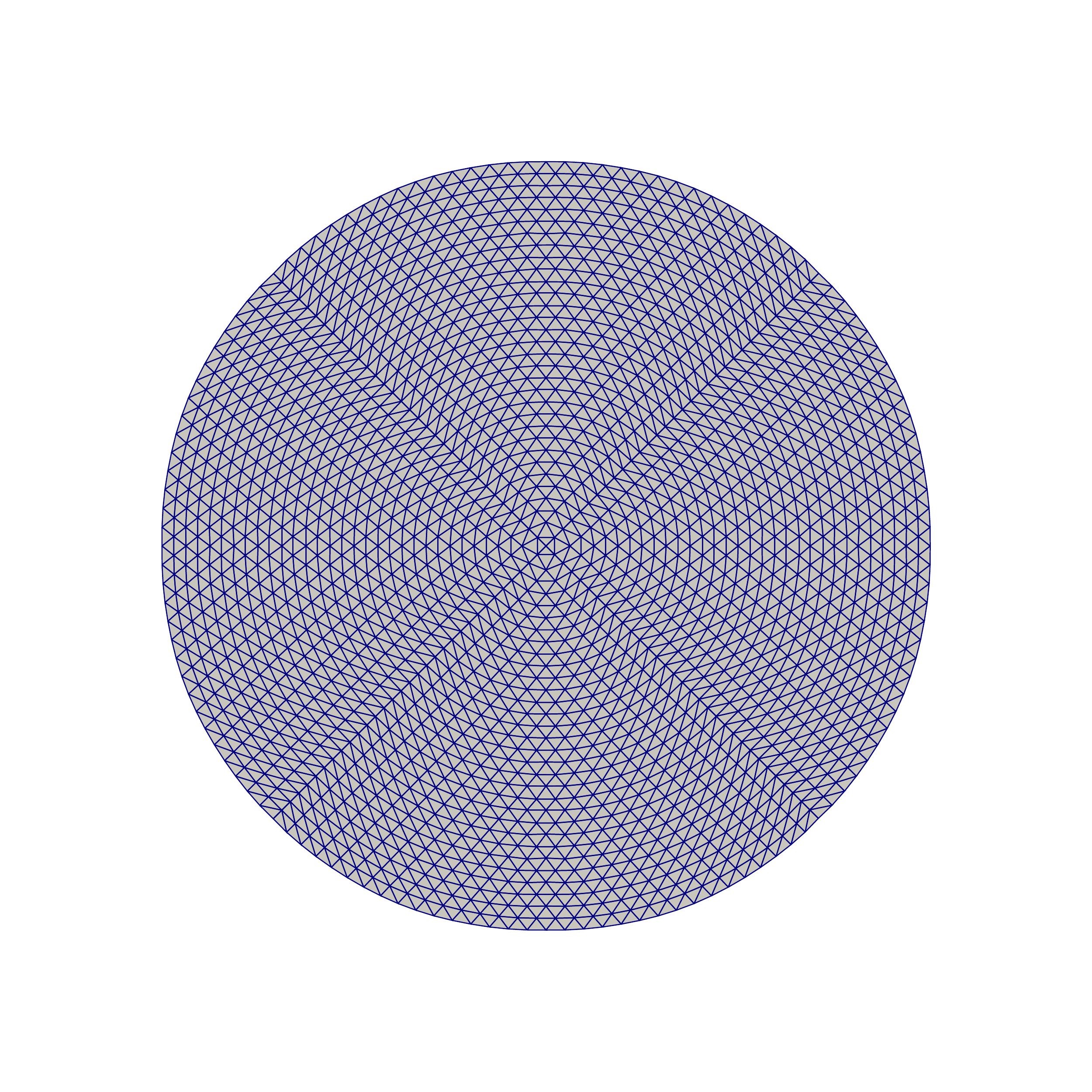}
\hspace{0.05\linewidth}
\includegraphics[width=.45\linewidth]{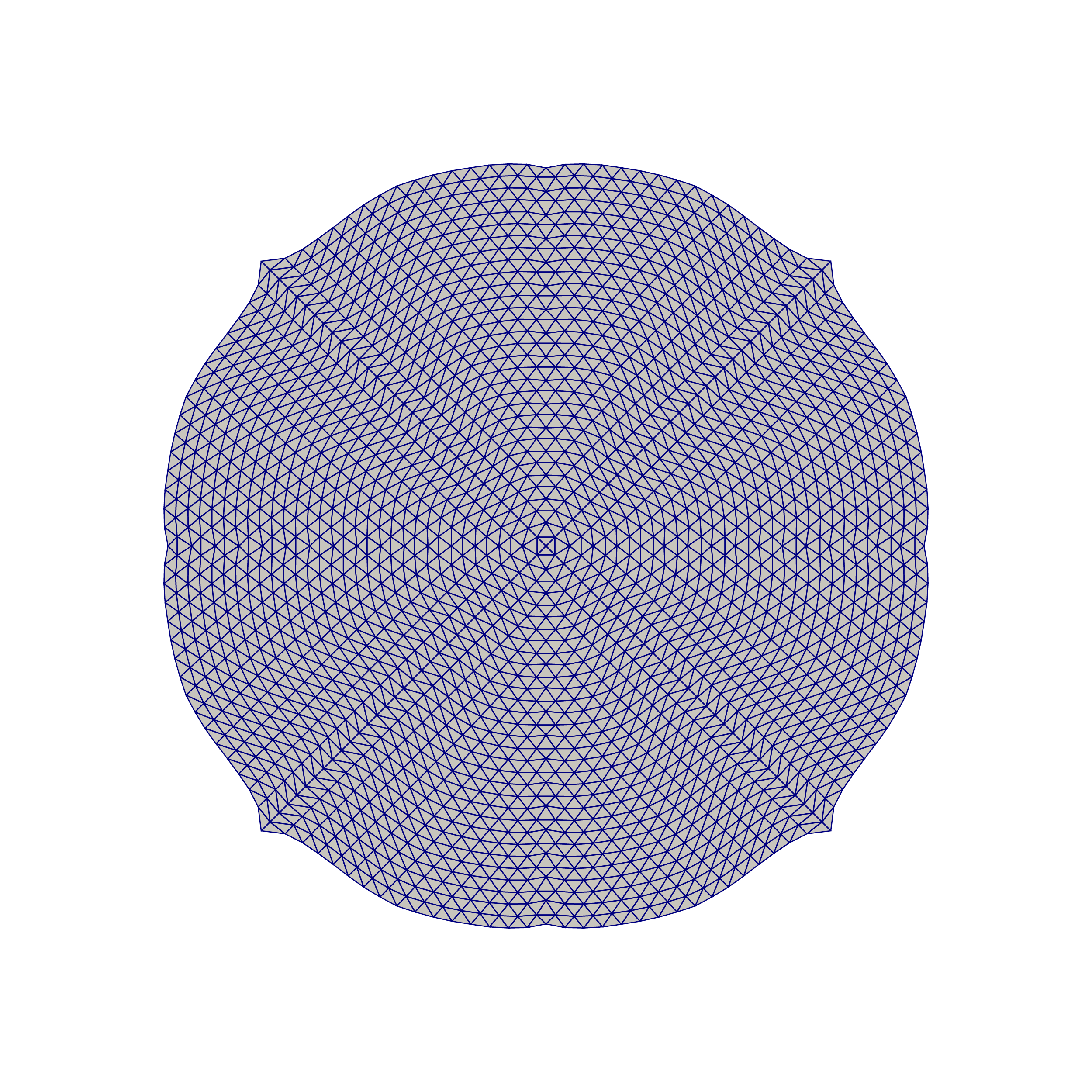}
\caption{Final domains for the experiment in Section \ref{sec:Experiment2} with Lipschitz optimal transport descent (left) and $H^1$ descent (right) with the volume form of the shape derivative.}
\label{fig:Experiment2FinalDomain}
\end{figure}
A graph of energy throughout the iterations is given in Figure \ref{fig:Experiment2Graph}.
\begin{figure}
    \centering
    \includegraphics[width=.9\linewidth]{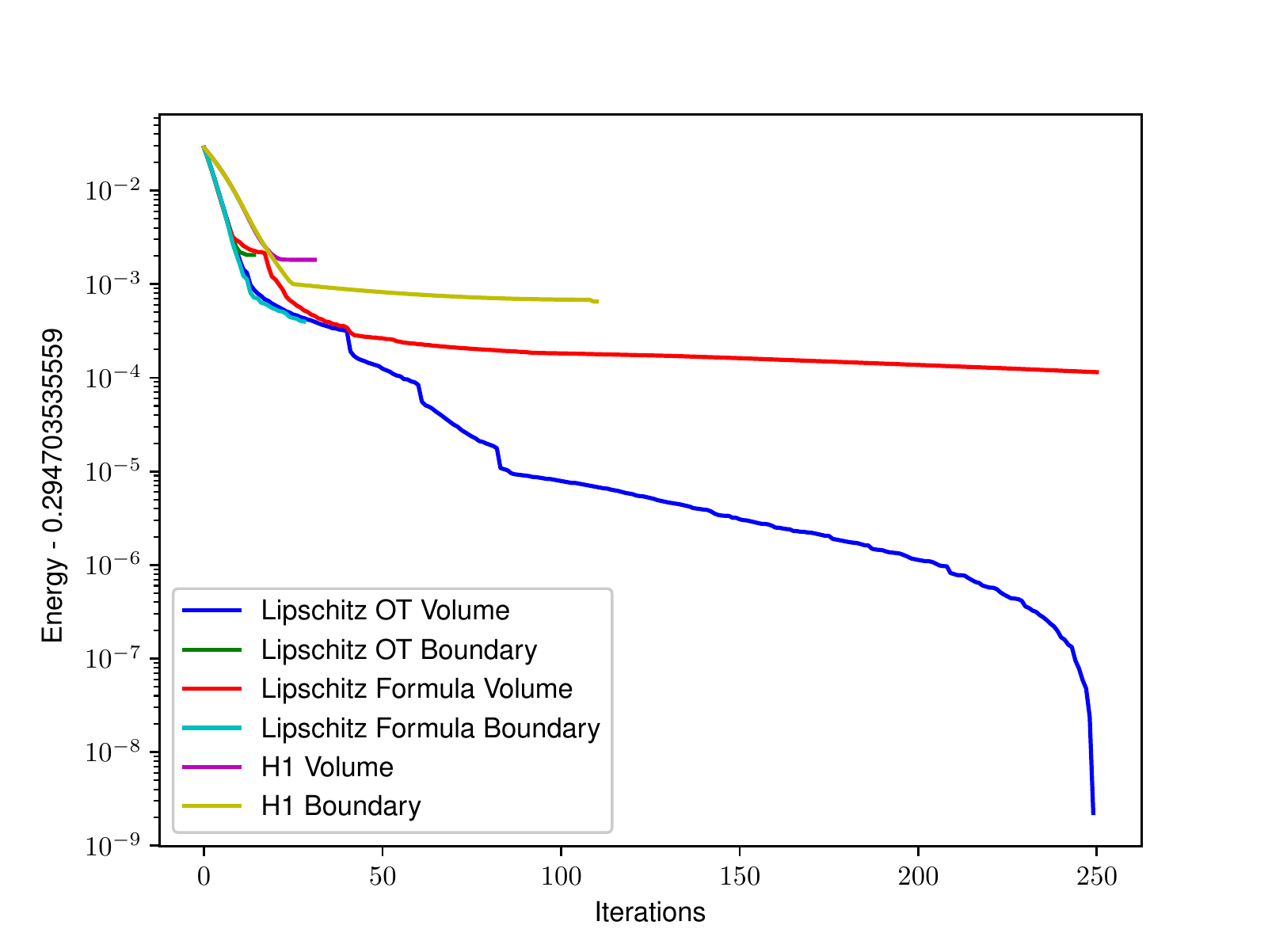}
    \caption{Graph of the energy for the iterates in the experiment in Section \ref{sec:Experiment2}}
    \label{fig:Experiment2Graph}
\end{figure}A graph of the magnitude of the discrete directional derivative throughout the iterations is given in Figure \ref{fig:Experiment2DiscreteDerivative}.
\begin{figure}
    \centering
    \includegraphics[width=.9\linewidth]{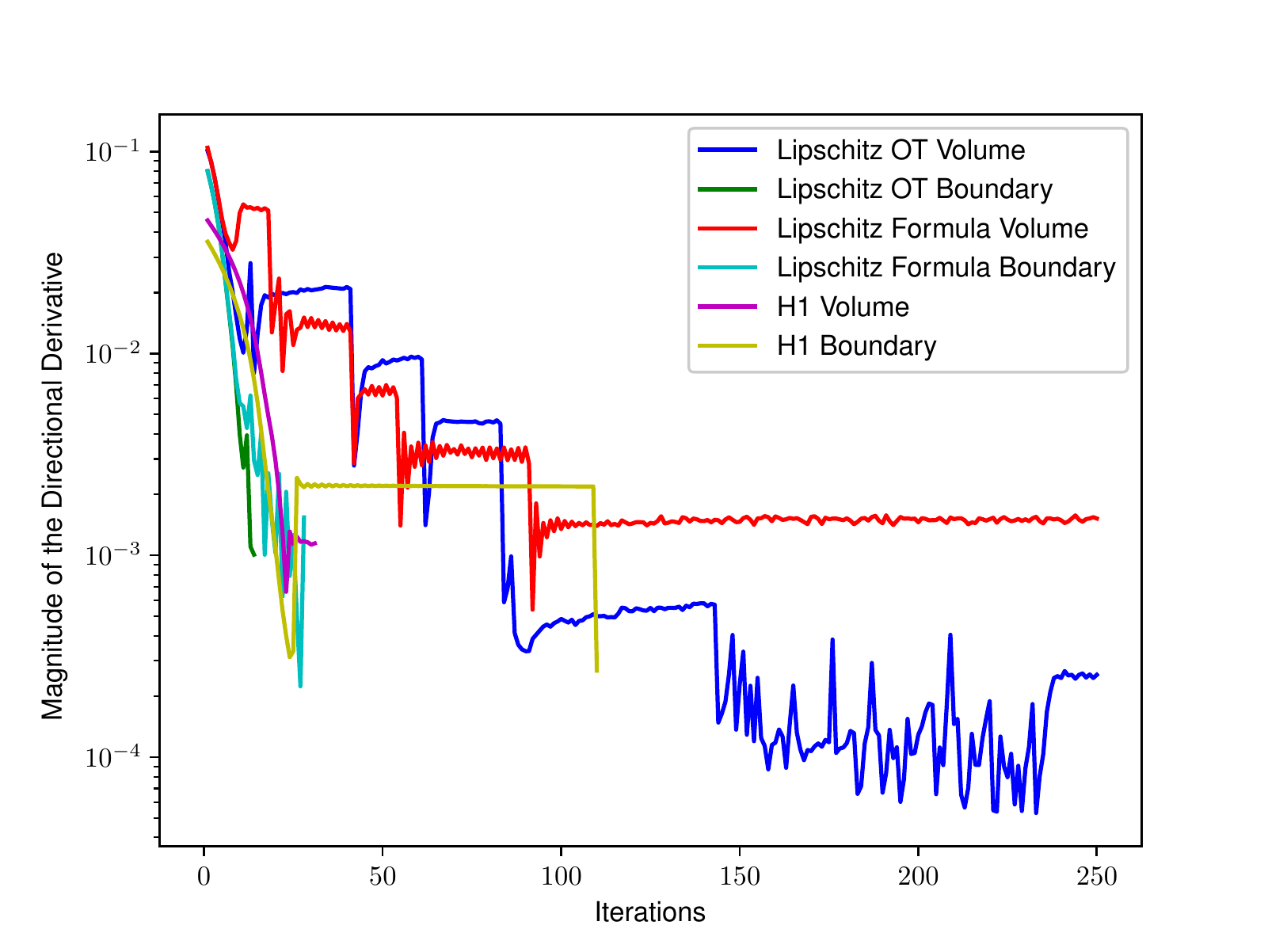}
    \caption{Graph of the magnitude of the discrete directional derivative for the iterates in the experiment in Section \ref{sec:Experiment2}}
    \label{fig:Experiment2DiscreteDerivative}
\end{figure}
We note that both $H^1$ methods terminate early, the boundary form of the shape derivative after $110$ iterations and the volume form of the shape derivative after $31$.
We postulate that this termination happens because the $H^1$ regularising methods are struggling to remove the corners.
We also see that both varieties of the Lipschitz methods with the boundary form of the shape derivative terminate prematurely.
The Lipschitz formula method terminated after $28$ steps and the Lipschitz optimal transport method after $14$.
One might attribute this to the Lipschitz methods struggling with the boundary form of the shape derivative.

It is seen that the corners appearing in the $H^1$ method, which are artifacts of the original grid, cause difficulties for the $H^1$ method, whereas the Lipschitz methods were able to remove them.
These artifact corners also make an appearance in \cite[Figure 2]{HP15} when starting with a square as initial guess and considering a disk as target.

\subsection{An experiment with $-\Delta z = F$} \label{sec:Experiment1}
For this experiment we set $F(x_1,x_2) = 16 \pi - 32x_1^2 -32x_2^2$ and $z(x_1,x_2) = (\pi-4x_1^2)(\pi-4x_2^2)$.
We notice that $-\Delta z = F$ and that $z(x_1, \pm \sqrt{\pi}/2) = z(\pm\sqrt{\pi}/2,x_2) = 0 $ for all $x_1,x_2 \in \R$.
An immediate consequence of these facts is that there is a domain which attains zero energy, the square $\left(-\frac{\sqrt{\pi}}{2}, \frac{\sqrt{\pi}}{2}\right)^2$.
The experiment is started with $f=1$.
After $250$ iterations, the Lipschitz optimal transport method with volume form of the shape derivative gives the domain
on the left of Figure \ref{fig:Experiment1FinalDomain} and after $250$ iterations the $H^1$ method with volume form of the shape derivative gives the domain on the right of Figure \ref{fig:Experiment1FinalDomain}.
\begin{figure}
\centering
\includegraphics[width=.45\linewidth]{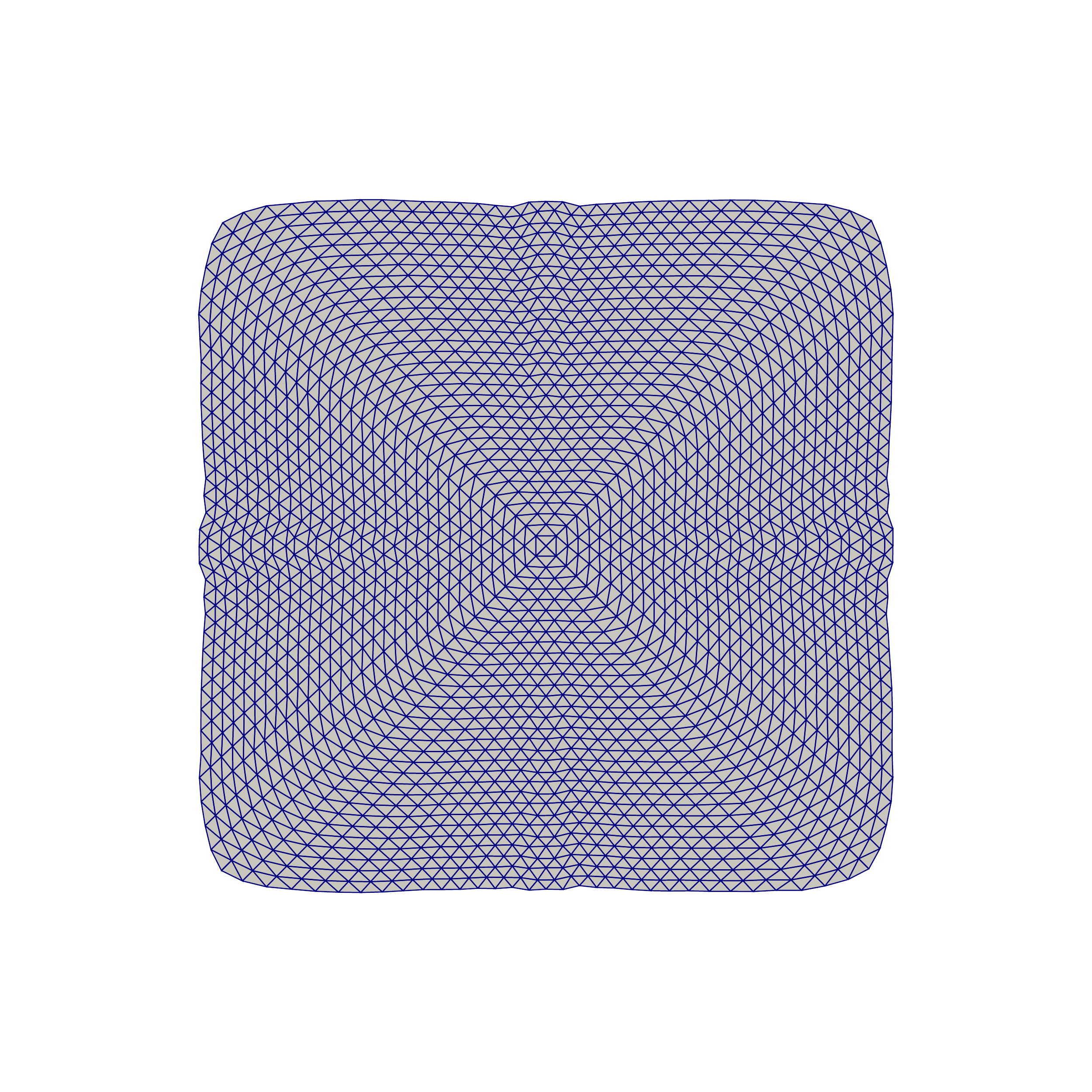}
\hspace{0.05\linewidth}
\includegraphics[width=.45\linewidth]{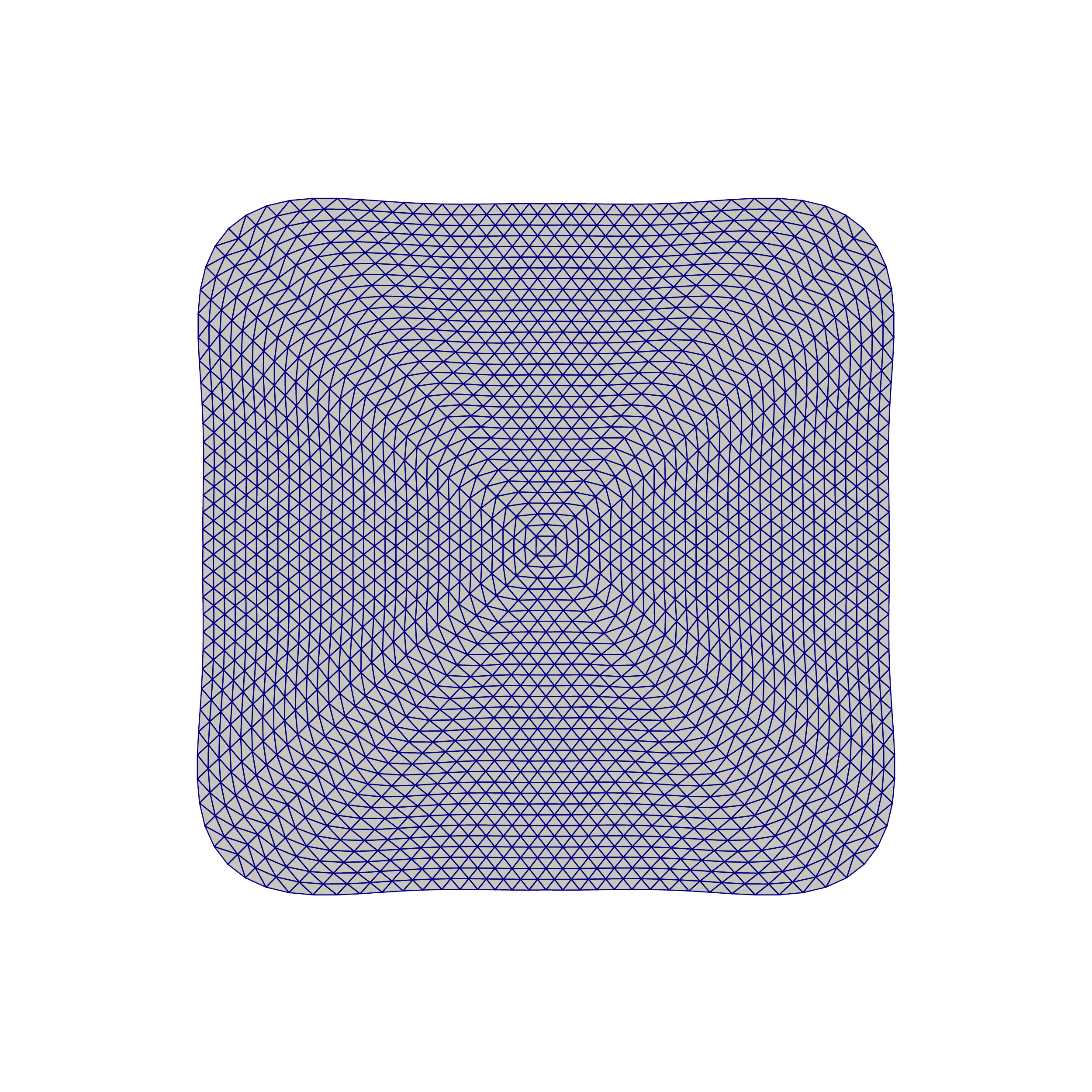}
\caption{Final domains for the experiment in Section \ref{sec:Experiment1} with Lipschitz optimal transport descent (left) and $H^1$ descent (right) with the volume form of the shape derivative.}
\label{fig:Experiment1FinalDomain}
\end{figure}
A graph of energy throughout the iterations is given in Figure \ref{fig:Experiment1Graph}.
\begin{figure}
    \centering
    \includegraphics[width=.9\linewidth]{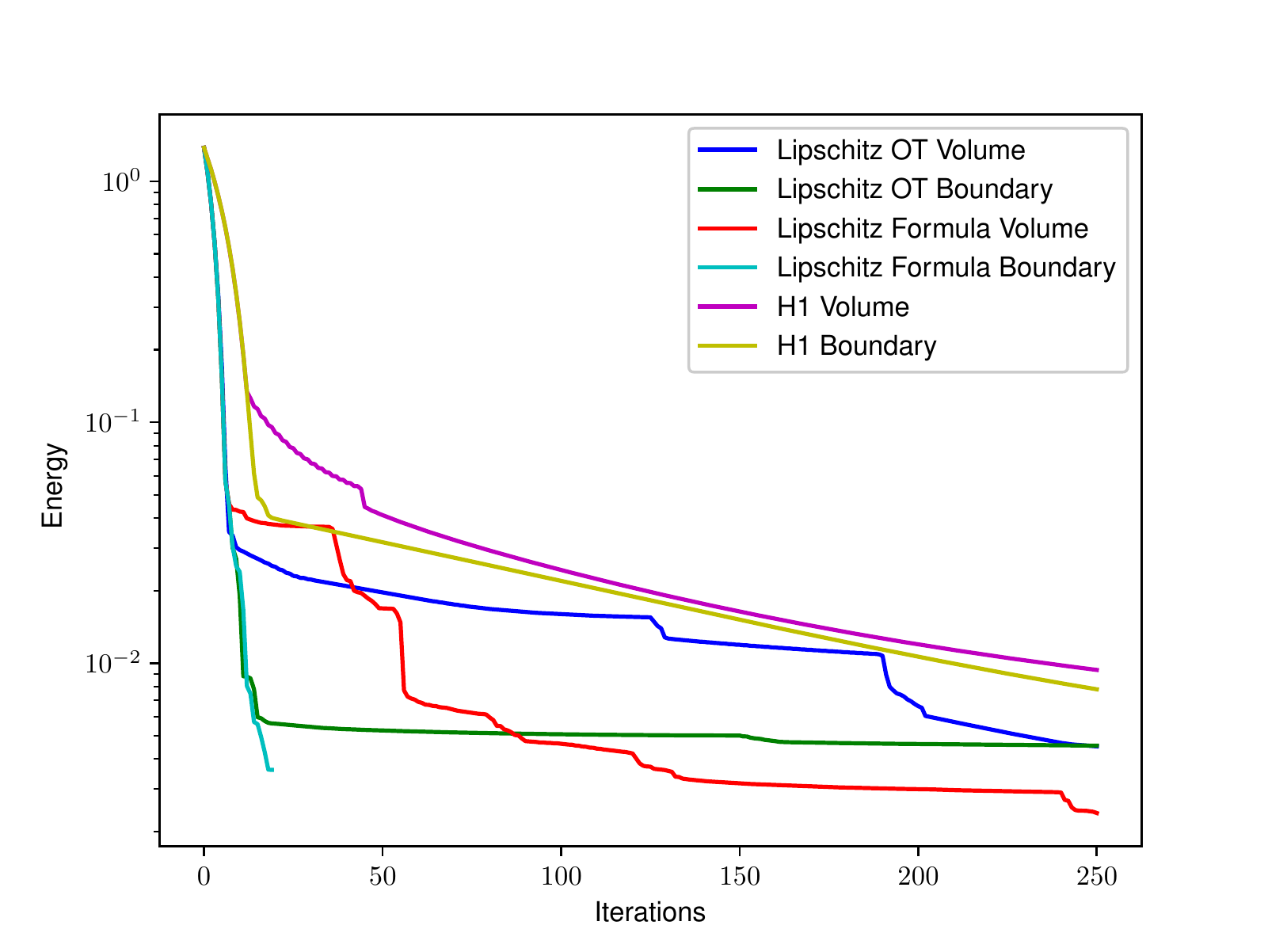}
    \caption{Graph of the energy for the iterates in the experiment in Section \ref{sec:Experiment1}}
    \label{fig:Experiment1Graph}
\end{figure}
A graph of the magnitude of the discrete directional derivative throughout the iterations is given in Figure \ref{fig:Experiment1DiscreteDerivative}.
\begin{figure}
    \centering
    \includegraphics[width=.9\linewidth]{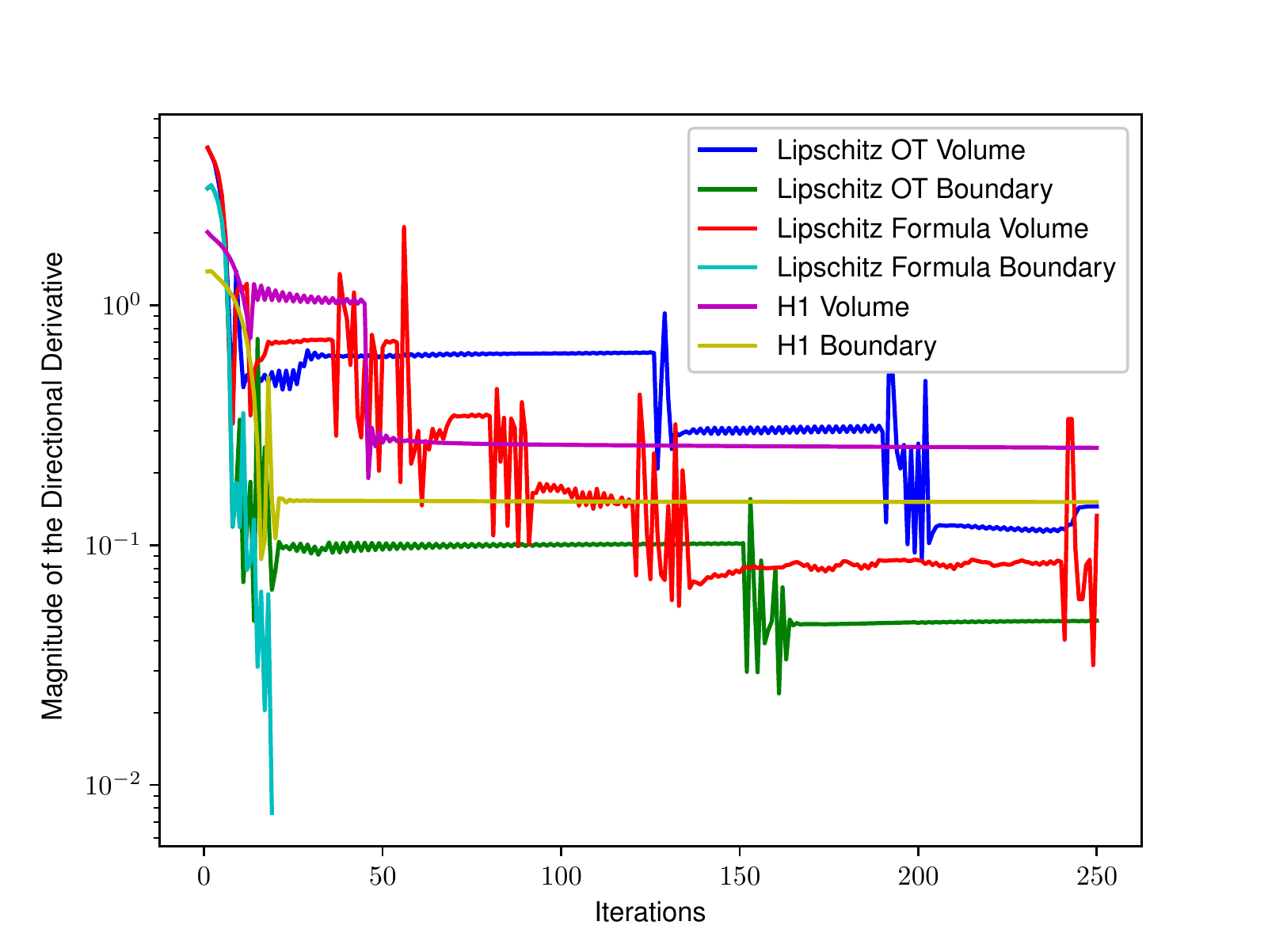}
    \caption{Graph of the magnitude of the discrete directional derivative for the iterates in the experiment in Section \ref{sec:Experiment1}}
    \label{fig:Experiment1DiscreteDerivative}
\end{figure}
It is seen that the method using the Lipschitz formula with the boundary form of the derivative terminates early, after $19$ iterations.
We note that this method, despite early termination, has a lower energy than all but one other method and we attribute the early termination to the fact that its energy has become so low.

Here we see that none of the methods perform particularly well, however it is clear that the Lipschitz methods are outperforming the $H^1$ methods in terms of energy minimisation and in terms of the sharpness of the corners.

\subsection{An experiment with known minimum which is not a Lipschitz domain} \label{sec:Experiment4}
For this experiment we set $F(x_1,x_2) = 1 $ and
\[
    z(x_1,x_2) = \frac{1}{8} - \frac{1}{4} \min \left( \left( x_1-\frac{1}{\sqrt{2}}\right)^2,\left( x_1+\frac{1}{\sqrt{2}} \right)^2 \right) - \frac{1}{4} x_2^2.
\]
We see that away from $x_1 = 0$, $-\Delta z = F$.
Therefore it is expected that the double ball, $B(y_+,\frac{1}{\sqrt{2}})\cup B(y_-,\frac{1}{\sqrt{2}})$ for $y_{\pm} = (\pm \frac{1}{\sqrt{2}},0)^t$ is a minimising domain, since it should attain zero energy.
Notice that this double ball is not a Lipschitz domain and that the $f$ which represents the domain has zeroes, therefore this experiment does not fit into the theory we have presented.
After $73$ iterations, the Lipschitz formula method with volume form of the derivative gives the domain on the left of Figure \ref{fig:Experiment4FinalDomain} and the $H^1$ method with volume form of the shape derivative gives the domain on the right of Figure \ref{fig:Experiment4FinalDomain}.
\begin{figure}
\centering
\includegraphics[width=.45\linewidth]{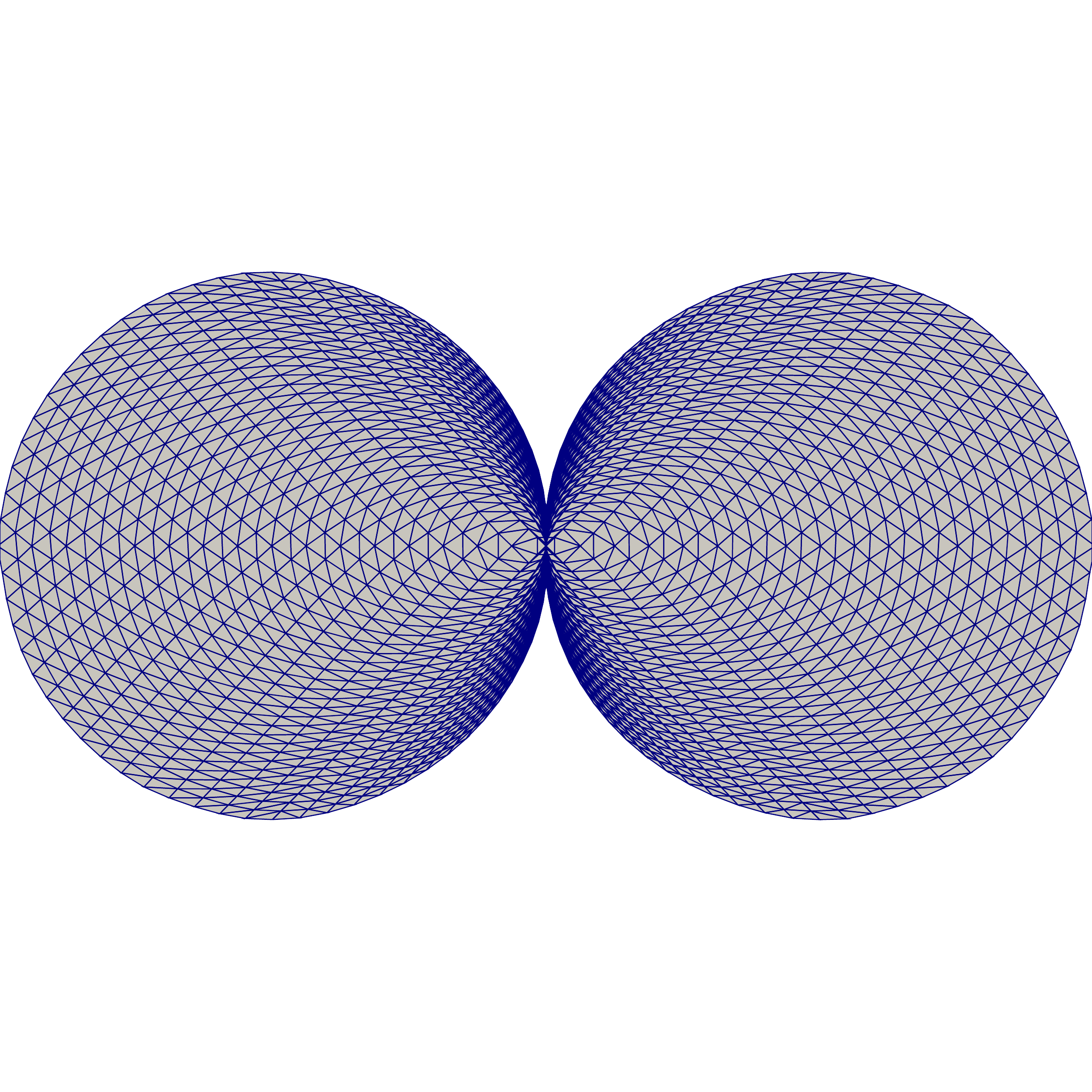}
\hspace{0.05\linewidth}
\includegraphics[width=.45\linewidth]{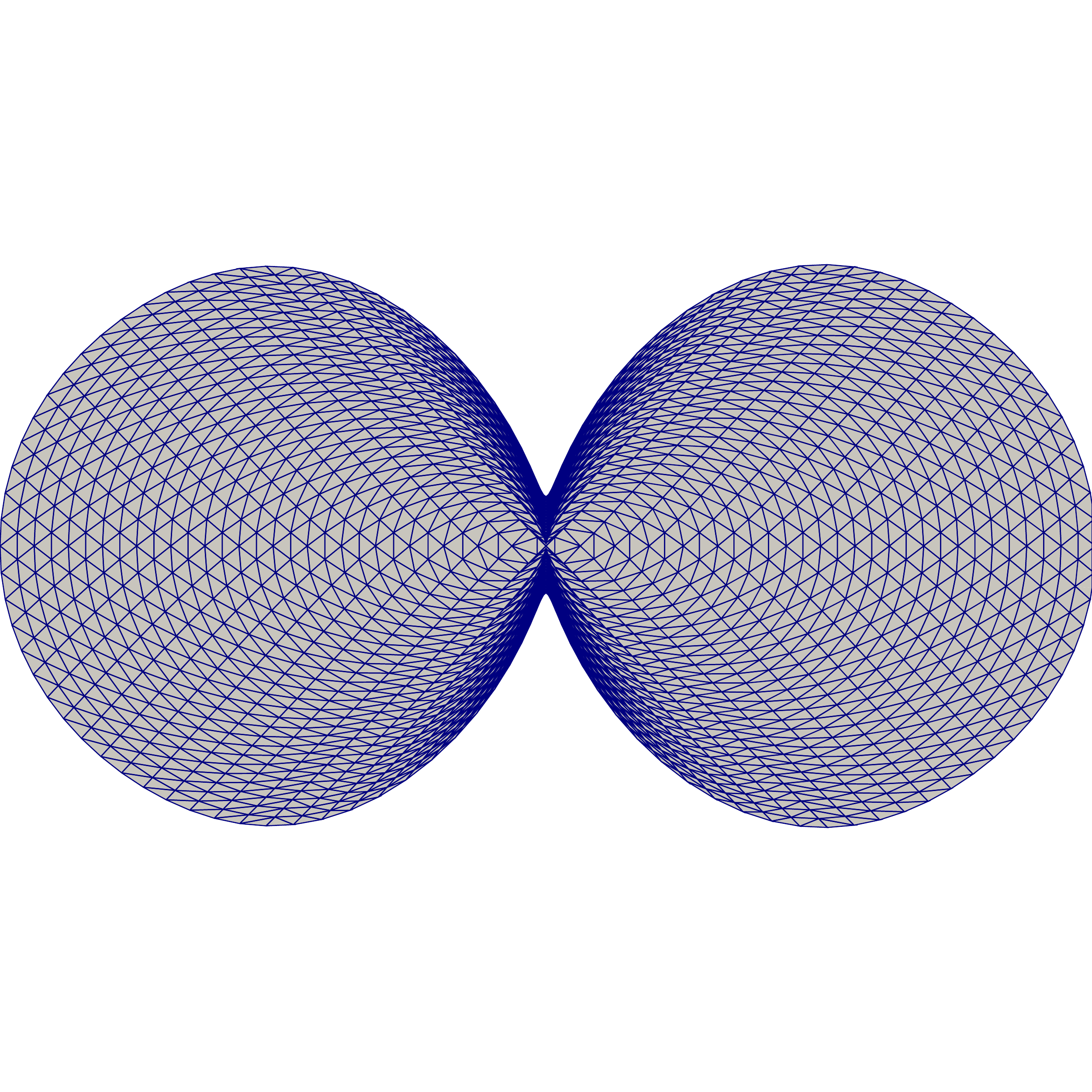}
\caption{Final domains for the experiment in Section \ref{sec:Experiment4} with Lipschitz formula descent (left) and $H^1$ descent (right) with the volume form of the shape derivative.}
\label{fig:Experiment4FinalDomain}
\end{figure}
A graph of energy throughout the iterations is given in Figure \ref{fig:Experiment4Graph}.
\begin{figure}
    \centering
    \includegraphics[width=.9\linewidth]{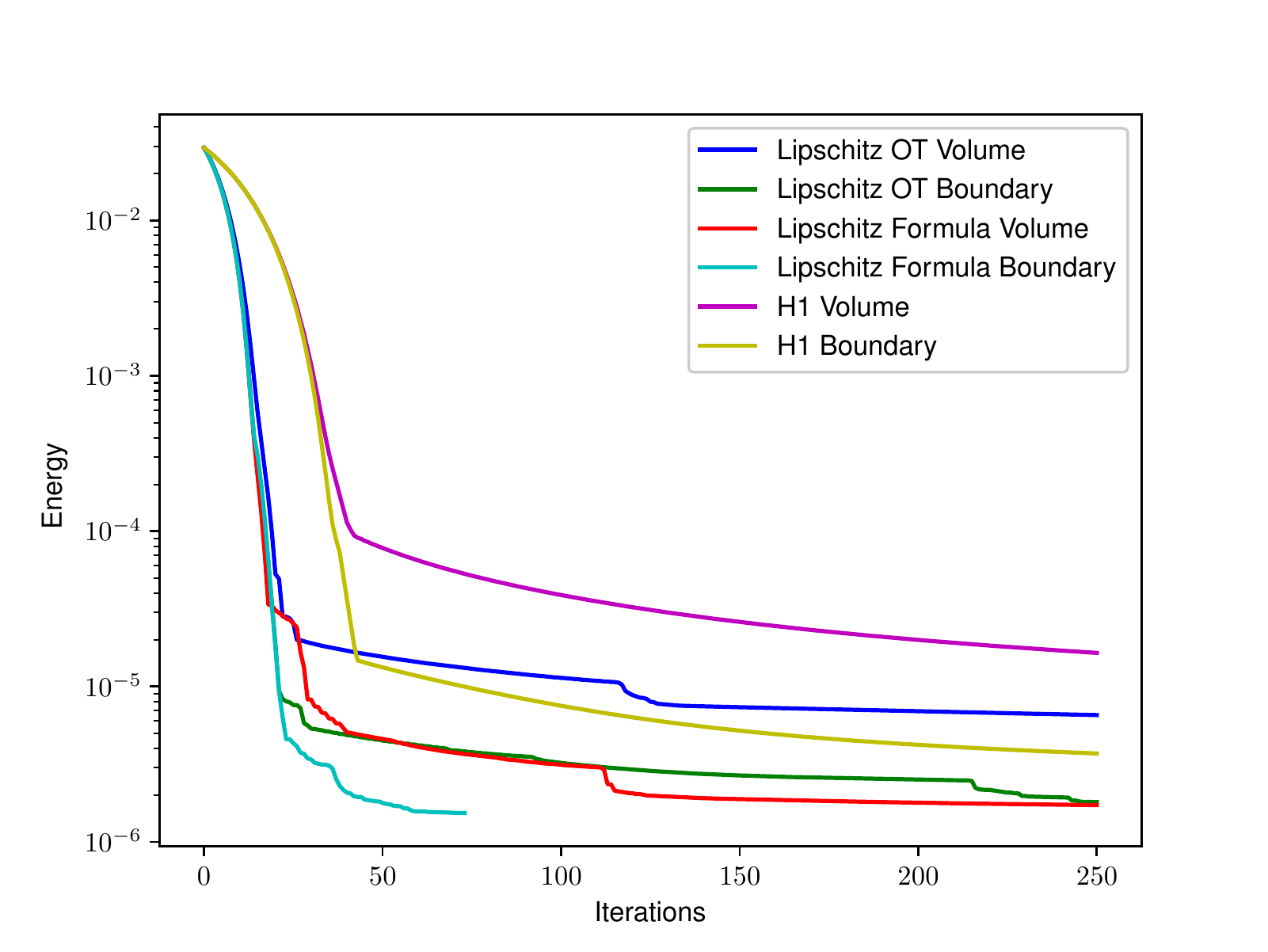}
    \caption{Graph of the energy for the iterates in the experiment in Section \ref{sec:Experiment4}}
    \label{fig:Experiment4Graph}
\end{figure}
A graph of the magnitude of the discrete directional derivative throughout the iterations is given in Figure \ref{fig:Experiment4DiscreteDerivative}.
\begin{figure}
    \centering
    \includegraphics[width=.9\linewidth]{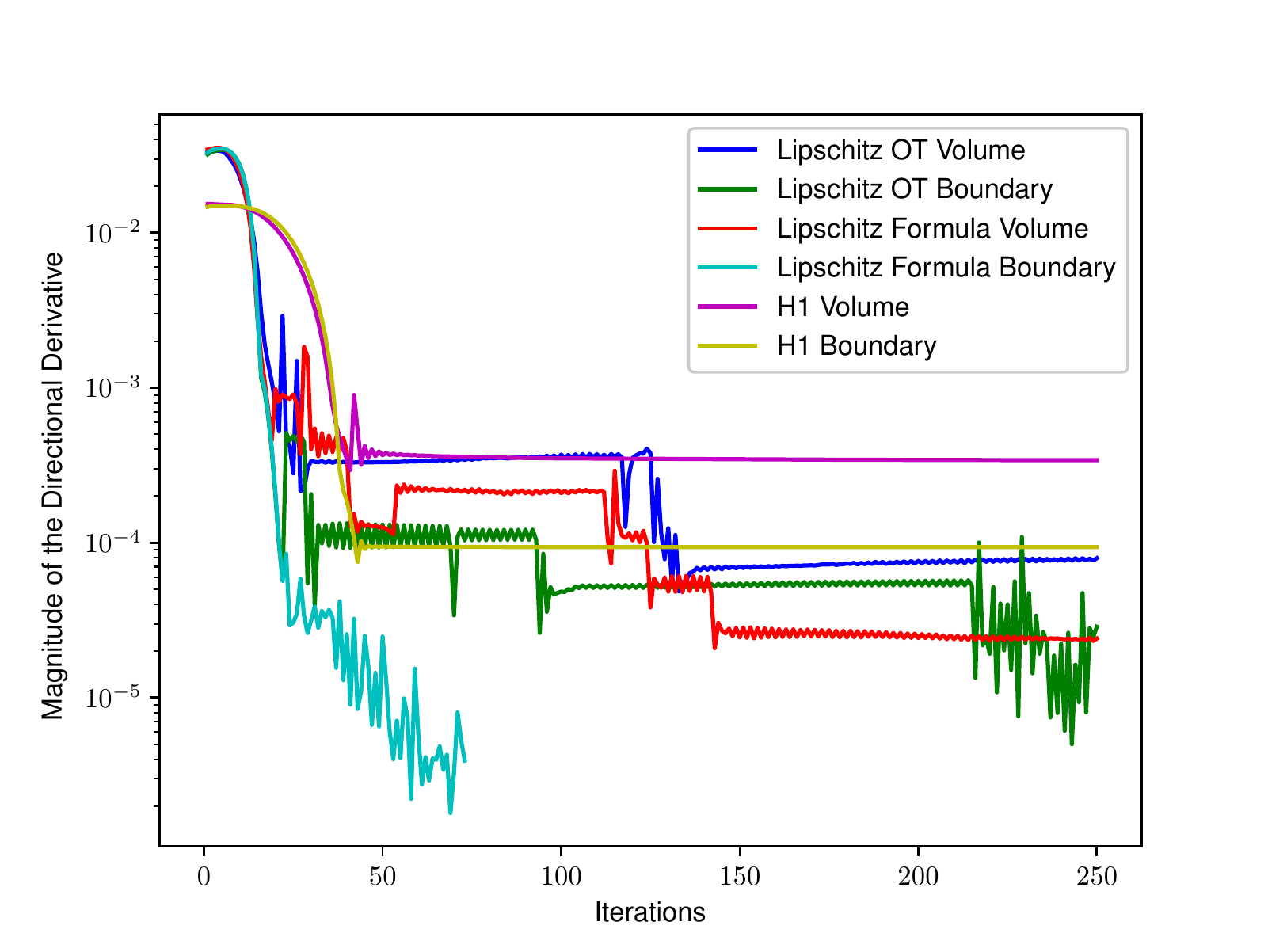}
    \caption{Graph of the magnitude of the discrete directional derivative for the iterates in the experiment in Section \ref{sec:Experiment4}}
    \label{fig:Experiment4DiscreteDerivative}
\end{figure}
We note that the Lipschitz formula method with volume form of derivative terminates after $73$ iterations, where one might attribute this to how close to the optimal shape it appears to have attained. 

We see that all methods seem to cope relatively well.
Both of the Lipschitz methods appear to perform much better than the $H^1$ method at forming the cusp and the domain appears more circular.

\subsection{Comments on experiments}
Over all of the experiments, we see that the Lipschitz methods outperform the given $H^1$ method.
Regularly for the Lipschitz approach, the formula method appears better than the optimal transport method, but lacks the capability to be generalised to higher dimensions.
We note that many more algorithms for solving the optimal transport are available and perhaps others may be better suited to this problem.

It is also worth mentioning the difference in CPU time it takes to calculate the directions for each of the different methods.
Both the Lipschitz methods are slower than the $H^1$ method for a single evaluation.
In the experiment which appears in Section \ref{sec:Experiment2}, we displayed some of the domains produced after $15$ iterations; we also calculated the time taken for the $15$ iterations.
We report only on the time taken when using the volume form of the derivative.
The Lipschitz optimal transport method took approximately $382$ seconds, the Lipschitz formula method roughly $87$ seconds and the $H^1$ method took about $68$ seconds.

With regards to these comparisons, it is noteworthy that the code has not been developed with efficiency in mind.
First of all, we expect that the parts which are not solver related can be made significantly faster.
In terms of finding the directions for descent, we expect that the Lipschitz optimal transport method can be made significantly more efficient in practice.

In practice, it is of course worth considering that one might wish to apply a number of iterations of a $H^1$ method in order to get a first outline of an optimised shape quickly, then swap to a more expensive Lipschitz method to finish.
This strategy may be of particular relevance in higher dimensional examples where one no longer has access to the formula approach.

\section{Conclusion}
In this article we introduce a novel method for the implementation of shape optimisation with Lipschitz domains. We propose to use the shape derivative to determine deformation fields which represent steepest descent directions of the shape functional in the $W^{1,\infty}$ topology. The idea of our approach is demonstrated for shape optimisation of two-dimensional star-shaped domains. We also highlight the connections to optimal transport, for which discretisation methods are available. We present several numerical experiments illustrating that our approach seems to be superior in the considered examples over the existing Hilbert space method discussed, in particular in developing optimal shapes with corners and in providing a quicker energy descent.

\section*{Acknowledgements}
This work is part of the project P8 of the German Research Foundation Priority Programme 1962, whose support is gratefully acknowledged by the second and the third author.

\printbibliography

\end{document}